\documentclass[a4paper,reqno,11pt]{amsart}

\usepackage{amsmath}
\usepackage{amsfonts}
\usepackage{palatino}
\usepackage{a4wide}
\usepackage{amssymb}
\usepackage{amsthm}
\usepackage{epsfig}
\usepackage{hyperref}
\usepackage{color}
\usepackage{caption}

\usepackage{graphicx,cite}
\usepackage{color}

\usepackage{tikz,fp,ifthen,fullpage}
\usetikzlibrary{backgrounds}
\usetikzlibrary{decorations.pathmorphing,backgrounds,fit,calc,through}
\usetikzlibrary{arrows}
\usetikzlibrary{shapes,decorations,shadows}
\usetikzlibrary{fadings}
\usetikzlibrary{patterns}
\usetikzlibrary{mindmap}
\usetikzlibrary{decorations.text}
\usetikzlibrary{decorations.shapes}

\usepackage[]{draftwatermark}
\SetWatermarkText{${\mathcal{DRAFT}}$}
\SetWatermarkScale{6.2}

\theoremstyle{plain}

\newtheorem{teo}{Theorem}[section]

\newtheorem{lemma}[teo]{Lemma}
\newtheorem{prop}[teo]{Proposition}

\newtheorem{ackn}{Acknowledgments\!}

\theoremstyle{definition}

\theoremstyle{remark}
\newtheorem{rem}[teo]{Remark}
\newtheorem{remark}[teo]{Remark}

\numberwithin{equation}{section}

\def\R{{{\mathbb R}}}
\def\SS{{{\mathbb S}}}

\newcommand{\pa}{\partial}
\newcommand{\mR}{\mathcal{R}}

\newcommand{\mE}{\mathcal{E}}
\newcommand{\mF}{\mathcal{F}}
\newcommand{\mG}{\mathcal{G}}

\newcommand{\mI}{\mathcal{I}}
\renewcommand{\d}{\delta}
\renewcommand{\a}{\alpha}
\renewcommand{\b}{\beta}
\newcommand{\s}{\sigma}
\renewcommand{\t}{\tau}
\newcommand{\g}{\gamma}

\newcommand{\N}{\mathbb{N}}
\renewcommand{\th}{\phi}

\begin{document}

\author[Pietro Baldi]{Pietro Baldi}
\address[Pietro Baldi]{Dipartimento di Matematica, Universit\`a di
  Napoli Federico II, Via Cintia, Monte S.~Angelo, Napoli, Italy,
  I--80126}
\email[P. Baldi]{pietro.baldi@unina.it}

\author[Emanuele Haus]{Emanuele Haus}
\address[Emanuele Haus]{Dipartimento di Matematica, Universit\`a di
  Napoli Federico II, Via Cintia, Monte S.~Angelo, Napoli, Italy,
  I--80126}
\email[E. Haus]{emanuele.haus@unina.it}

\author[Carlo Mantegazza]{Carlo Mantegazza}
\address[Carlo Mantegazza]{Dipartimento di Matematica, Universit\`a di
  Napoli Federico II, Via Cintia, Monte S.~Angelo, Napoli, Italy,
  I--80126}
\email[C. Mantegazza]{c.mantegazza@sns.it}

\title[Non--Existence of {\em Theta}--Shrinkers]{Non--Existence of
  {\em Theta}--Shaped Self--Similarly Shrinking Networks Moving by Curvature}
\date{\today}

\maketitle
\setcounter{tocdepth}{1}
\tableofcontents

\begin{abstract} We prove that there are no networks homeomorphic to
  the Greek ``theta'' letter (a double cell) embedded in the plane
  with two triple junctions with angles of $120$ degrees, such that
  under the motion by curvature they are self--similarly shrinking.\\
This fact completes the classification of the self--similarly
shrinking networks in the plane with at most two triple
junctions, see~\cite{chenguo,haettenschweiler,schnurerlens}.
\end{abstract}

\section{Introduction}

Recently, the problem of the evolution by curvature of a network of
curves in the plane got the interest of several
authors~\cite{mannovtor,BeNo,schn-schu,schnurerlens,haettenschweiler,mazsae,Ilnevsch,MMN13,pluda}.
It is well known that, after the work of Huisken~\cite{huisk3} in the
smooth case of the hypersurfaces in the Euclidean space and of
Ilmanen~\cite{ilman1,ilman3} in the more general weak settings of
varifolds, that a suitable sequence of rescalings of the {\em subsets}
of $\R^n$ which are evolving by mean curvature, approaching a singular
time of the flow, converges to a so called ``blow--up limit'' set
which, letting it flow again by mean curvature, simply moves by
homothety, precisely, it shrinks down self--similarly toward the origin of
the Euclidean space.

This procedure and the classification of these special sets (possibly
under some hypotheses), called {\em shrinkers}, is a key point in
understanding the asymptotic behavior of the flow at a singular time.

In the special situation of a network of curves in the plane moving by curvature (see~\cite{mannovtor,mannovplusch}, for instance), it is easy to see that every $C^2$ curve $\gamma:I\to\R^2$ of such a limit shrinking network must satisfy the following ``structural'' equation (which is actually an ODE for $\gamma$)
\begin{equation} \label{1102.1}
\overline{k}+\gamma^\perp=0,
\end{equation}
where $\overline{k}$ is the vector curvature of the curve at the point
$\gamma$ and $\gamma^\perp$ denotes the normal component of the
position vector $\gamma$. Introducing an arclength parameter $s$ on
the curve $\gamma$, we have a unit tangent vector field
$\tau=\frac{d\,}{ds}\gamma$, a unit normal vector field $\nu$ which is
the counterclockwise rotation of $\pi/2$ in $\R^2$ of the vector
$\tau$ and the curvature vector given by
$\overline{k}=k\nu=\frac{d^2\,}{ds^2}\gamma$, where $k$ is then simply
the curvature of $\gamma$. With these notations, the above equation
can be rewritten as 
\begin{equation} \label{1102.1-bis}
k+\langle\gamma\,\vert\,\nu\rangle=0.
\end{equation}
It is known, by the work of Abresch--Langer~\cite{ablang1} and independently of
Epstein--Weinstein~\cite{epswei}, that the only complete, embedded
shrinking curves in $\R^2$ are the lines through the origin and the
unit circle.

Dealing with the motion by curvature of networks of curves in the
plane, it is natural (for ``energy reasons'') to consider networks with
only triple junctions and such that the three concurring curves (which
are $C^\infty$) form three angles of $120$ degrees between each other
(``Herring'' condition); such networks are called {\em
  regular}. Analogously, a shrinker is called regular if it satisfies
such geometric condition.

If one consider networks with only one triple junction, the only
complete, embedded, connected, regular shrinkers are given (up to a
rotation) by the ``standard triod'' and the ``Brakke spoon'' (first
described in~\cite{brakke}) as in the following figures.

\begin{figure}[h]
\begin{center}
\begin{tikzpicture}[scale=0.6]
\draw[shift={(-10,0)}]
(-2,-2) to [out=45, in=-135, looseness=1] (2,2);
\draw[dashed,shift={(-10,0)}]
(-2.5,-2.5) to [out=45, in=-135, looseness=1] (-2,-2)
(2,2)to [out=45, in=-135, looseness=1](2.6,2.6);
\draw[shift={(+10,0)}]
(0,0) to [out=90, in=-90, looseness=1] (0,2)
(0,0) to [out=210, in=30, looseness=1] (-1.73,-1)
(0,0) to [out=-30, in=150, looseness=1](1.73,-1);
\draw[dashed,shift={(+10,0)}]
(0,2) to [out=90, in=-90, looseness=1] (0,3)
(-1.73,-1) to [out=210, in=30, looseness=1] (-2.59,-1.5)
(1.73,-1) to [out=-30, in=150, looseness=1](2.59,-1.5);
\draw[color=black,scale=1,domain=-4.141: 4.141,
smooth,variable=\t,shift={(0,0)},rotate=0]plot({2.*sin(\t r)},
{2.*cos(\t r)}) ;
\fill(0,0) circle (2pt);
\fill(-10,0) circle (2pt);
\fill(10,0) circle (2pt);
\path[font=\small]
(-10,.2) node[left]{$O$}
(10,.2) node[left]{$O$}
(.1,.2) node[left]{$O$};
\end{tikzpicture}
\end{center}
\begin{caption}{Easy examples of regular shrinkers: a line from the
    origin, the unit circle $\SS^1$ and an unbounded triod composed of
    three halflines from the origin meeting at $120$ degrees, called
    standard triod.}
\end{caption}
\end{figure}

\begin{figure}[h]
\begin{center}
\begin{tikzpicture}[scale=0.4]
\draw[color=black]
(-4.8,0)to[out= 0,in=180, looseness=1](-1.535,0)
(-1.535,0)to[out= 60,in=180, looseness=1] (3.7,3)
(3.7,3)to[out= 0,in=90, looseness=1] (6.93,0)
(-1.535,0)to[out= -60,in=180, looseness=1] (3.7,-3)
(3.7,-3)to[out= 0,in=-90, looseness=1] (6.93,0);
\draw[color=black,dashed](-9,0)to[out= 0,in=180, looseness=1](-5.0,0);
\fill(1,0) circle (3pt);
\path[font=\large]
 (.3,-.35) node[above]{$O$};
\end{tikzpicture}
\end{center}
\begin{caption}{A less easy example of a regular shrinker: a Brakke spoon.}
\end{caption}
\end{figure}

About shrinkers with two triple junctions, it is not difficult to show
that there are only two possible topological shapes for a complete
embedded, regular shrinker: one is the ``lens/fish'' shape and the
other is the shape of the Greek ``theta'' letter (or ``double cell''),
as in the next figure.

\begin{figure}[h]
\begin{center}
\begin{tikzpicture}[scale=1]
\draw[shift={(0,0)}]  
(-3.73,0)
to[out= 50,in=180, looseness=1] (-2.8,0) 
to[out= 60,in=150, looseness=1.5] (-1.5,1) 
(-2.8,0)
to[out=-60,in=180, looseness=0.9] (-1.25,-0.75)
(-1.5,1)
to[out= -30,in=90, looseness=0.9] (-1,0)
to[out= -90,in=60, looseness=0.9] (-1.25,-0.75)
to[out= -60,in=180, looseness=0.9](-0.3,-1.3);
\draw[dashed,shift={(0,0)}]  
(-3.73,0)to[out= 50,in=180, looseness=1] (-4,-.6);
\draw[dashed,shift={(0,0)}]  
(-0.3,-1.3)to[out= 0,in=150, looseness=.8] (0.53,-.8);
\path[font=\small,shift={(0,0)}]
(-2.4,0.3) node[below] {$O_1$}
(-1.25,-0.75)node[right]{$O_2$}
(-1.5,-1)[left] node{$\gamma^2$}
(-3.2,.4) node[left] {$\gamma^1$}
(-0.5,0.65)[left] node{$\gamma^4$}
(-0.35,-1.30)[below] node{$\gamma^3$};
\draw[shift={(5,0)}] 
(-1.73,-1.8) 
to[out= 180,in=180, looseness=1] (-2.8,0) 
to[out= 60,in=150, looseness=1.5] (-1.5,1) 
(-2.8,0)
to[out=-60,in=180, looseness=0.9] (-1.25,-0.75)
(-1.5,1)
to[out= -30,in=90, looseness=0.9] (-1,0)
to[out= -90,in=60, looseness=0.9] (-1.25,-0.75)
to[out= -60,in=0, looseness=0.9](-1.73,-1.8);
\path[font=\small, shift={(5,0)}] 
(-1.25,-0.85)node[right]{$O_2$}
 (-1.5,-0.3)[left] node{$\gamma^2$}
 (-0.6,.9)[left] node{$\gamma^1$}
 (-0.6,-1.45)[left] node{$\gamma^3$}
 (-3,0.65) node[below] {$O_1$};   
\end{tikzpicture}
\end{center}
\begin{caption}{A lens/fish--shaped and a $\Theta$--shaped network.}
\end{caption}\end{figure}

It is well known that there exist unique (up to a rotation)
lens--shaped or fish--shaped, embedded, regular shrinkers which are
symmetric with respect to a line through the origin of $\R^2$
(see~\cite{chenguo,schnurerlens}). It was instead 
unknown whether regular $\Theta$--shaped shrinkers (or simply $\Theta$--shrinkers) exist, with numerical evidence
in favor of the conjecture of non--existence
(see~\cite{haettenschweiler}). In this paper we are going to show that
this is actually the case.

\begin{figure}[h]
\begin{center}
\begin{tikzpicture}[scale=0.3]
\draw[color=black]
(-3.035,0)to[out= 60,in=180, looseness=1] (2.2,2.8)
(2.2,2.8)to[out= 0,in=120, looseness=1] (7.435,0)
(2.2,-2.7)to[out= 0,in=-120, looseness=1] (7.435,0)
(-3.035,0)to[out= -60,in=180, looseness=1] (2.2,-2.7);
\draw[color=black]
(-7,0)to[out= 0,in=180, looseness=1](-3.035,0)
(7.435,0)to[out= 0,in=180, looseness=1](11.4,0);
\draw[color=black,dashed]
(-9,0)to[out= 0,in=180, looseness=1](-7,0)
(11.4,0)to[out= 0,in=180, looseness=1](13.4,0);
\fill(2.2,0) circle (3pt);
\path[font=\small]
(1.5,-.35) node[above]{$O$};
\draw[color=black,scale=4,shift={(7,0)}]
(-0.47,0)to[out= 20,in=180, looseness=1](1.5,0.65)
(1.5,0.65)to[out= 0,in=90, looseness=1] (2.37,0)
(-0.47,0)to[out= -20,in=180, looseness=1](1.5,-0.65)
(1.5,-0.65)to[out= 0,in=-90, looseness=1] (2.37,0);
\draw[white, very thick,scale=4,shift={(7,0)}]
(-0.47,0)--(-.150,-0.13)
(-0.47,0)--(-.150,0.13);
\draw[color=black,scale=4,shift={(7,0)}]
(-.150,0.13)to[out= -101,in=90, looseness=1](-.18,0)
(-.18,0)to[out= -90,in=101, looseness=1](-.150,-0.13);
\draw[black,scale=4,shift={(7,0)}]
(-.150,0.13)--(-1.13,0.98);
\draw[black,scale=4,shift={(7,0)}]
(-.150,-0.13)--(-1.13,-0.98);
\draw[black, dashed,scale=4,shift={(7,0)}]
(-1.13,0.98)--(-1.50,1.31);
\draw[black, dashed,scale=4,shift={(7,0)}]
(-1.13,-0.98)--(-1.50,-1.31);
\fill(28.05,0) circle (3pt);
\path[font=\small]
(29,-.6) node[above]{$O$};
\end{tikzpicture}
\end{center}
\begin{caption}{A lens--shaped and a fish--shaped shrinker.}
\end{caption}
\end{figure}

\begin{figure}[h]
\begin{center}
\begin{tikzpicture}[scale=0.5, rotate=90]
\draw[color=black]
(-3.035,0)to[out= 60,in=180, looseness=1] (2.2,3)
(2.2,3)to[out= 0,in=90, looseness=1] (5.43,0)
(-3.035,0)to[out= -60,in=180, looseness=1] (2.2,-3)
(2.2,-3)to[out= 0,in=-90, looseness=1] (5.43,0);
\draw[color=black, rotate=180]
(-3.035,0)to[out= 60,in=180, looseness=1] (2.2,3)
(2.2,3)to[out= 0,in=90, looseness=1] (5.43,0)
(-3.035,0)to[out= -60,in=180, looseness=1] (2.2,-3)
(2.2,-3)to[out= 0,in=-90, looseness=1] (5.43,0);
\draw[color=white, very thick]
(-3.035,0)to[out= -60,in=157.5, looseness=1] (-0.05,-2.65)
(-3.035,0)to[out= 60,in=-157.5, looseness=1] (-0.05,2.65);
\draw[color=white, very thick]
(3.035,0)to[out=-120,in=22.5, looseness=1] (0.05,-2.65)
(3.035,0)to[out= 120,in=-22.5, looseness=1] (0.05,2.65);
\draw[color=black]
(0,2.4)to[out= 90,in=-90, looseness=1](0,-2.4);
\fill(0,0) circle (3pt);
\path[font=\small]
(0,.4) node[above]{$O$};
\end{tikzpicture}
\end{center}
\begin{caption}{A hypothetical $\Theta$--shrinker.}
\end{caption}
\end{figure}

\begin{teo} \label{thm:main}
There are no regular $\Theta$--shrinkers.
\end{teo}

The main motivation for this problem is given by the fact that for an evolving
network with at most {\em two} triple junctions the so called {\em
  multiplicity--one conjecture} holds (see~\cite{mannovplusch}),
saying that any limit shrinker of a sequence of rescalings of the network at
different times is again a ``genuine'' embedded network without ``double'' or
``multiple'' curves (curves that in such convergence go to
coincide). This is a key point in singularity analysis (in general, for
mean curvature flow), together with the classification of these limit
shrinkers, which is complete after our result Theorem~\ref{thm:main}, for such ``low complexity'' networks, thus leading to a full
description of their motion in~\cite{mannovplu}. Moreover, in the general case, the ``generic'' singularity should be (locally) the collapse of
a single curve (only two triple junction colliding), hence the study of the evolution of
networks with only two triple junctions is quite interesting since
they locally describe what happens at a ``typical'' singular
time. 

The line to show Theorem~\ref{thm:main} will be first looking at the
properties that an hypothetical $\Theta$--shrinker must satisfy and
study the possible geometric structures that {\em a priori} it could
have (Sections~\ref{sec:Cp} and~\ref{sec:pof}). Such analysis will
reduce the proof of non--existence to show that a certain parametric
integral is always smaller than $\pi/2$, for every value of the
parameter. Section~\ref{sec:proof of lemma zucca} is devoted to show
such an estimate, mixing some approximation techniques and numerical
computations based on {\em interval arithmetic}.

\subsection{Interval arithmetic}

When performing standard analytical computations on a machine, the
computer routines repeatedly introduce unpleasant errors. This cannot
be avoided since, for instance:
\begin{itemize}
\item the real numbers that a machine can handle exactly in
  floating--point binary representation are only a finite subset of
  the rationals;
\item transcendental functions cannot be computed exactly and have to
  be replaced by an approximation based on a finite Taylor expansion.
\end{itemize}
While these limitations make it impossible to get an exact result from
-- basically -- any numerical computation on a machine, it is however
possible to guarantee that the value of a given analytical expression
is included between two values that can be represented \emph{exactly}
by the computer, providing rigorous upper and lower bounds.

The term ``interval arithmetic'' refers to the arithmetic of
quantities whose value lies within a known interval, although the
exact value of such quantities is not known. The simple considerations
presented above clearly demonstrate why the theory and implementation
of interval arithmetic is a fundamental tool for making rigorous
computer--assisted proofs in mathematical analysis.

\begin{rem}
Interval arithmetic for rigorous computer--assisted proofs has been
used in several areas of mathematics (see, for example, the
reviews~\cite{Moore91} and~\cite{FefSec96}).
See~\cite{Fousse2007} for documentation about correct rounding of elementary functions. 
Note that all the computations of the present paper that we made with the computer 
involve only elementary functions (i.e. $x+y$, $x \cdot y$, $x^2$,
$1/x$, $\sqrt{x}$, $\exp(x)$, $\log(x)$, $\arcsin(x)$) and constants
($\pi$) whose correct rounding is guaranteed by the IEEE Standard for
Interval Arithmetic, IEEE~1788--2015~\cite{IEEE1788}. 
To perform interval arithmetic, we used the software GNU
Octave 4.0.0 with the package ``interval'', which is conforming
(see~\cite{conforming}) to the standard IEEE~1788--2015.

We give a basic example of the use of {\em interval arithmetic}: 
to deal with the constant $\sqrt{3}$, 
we write the command ``sqrt(infsup($3$))'': writing ``infsup($3$)''
the integer number $3$ is transformed into the interval $[3,3]$ (which
is just the singleton $\{ 3 \}$ since the integer number $3$ is
exactly representable by the computer), 
then the function ``square root'' applies to the interval $[3,3]$ and
it gives a correct rounding of the exact mathematical value of
$\sqrt{3}$, namely the output is an interval $[a,b]$ such that $a^2 <
3 < b^2$ and $a,b$ are exactly representable in binary form.

The codes we have used are in Section~\ref{sec:codes}.
\end{rem}

\medskip

\begin{ackn} We wish to thank Matteo~Novaga and Alessandra~Pluda for
  several discussions on the subject of this paper.\\
This research was financially supported by UniNA and Compagnia di
San~Paolo, in the frame of Programme~STAR, by the European Research Council
under FP7 (ERC~Project 306414) and by PRIN~2012 ``Variational and perturbative aspects of nonlinear
differential problems''.
\end{ackn}

\section{Basic properties of shrinking curves and notation}
\label{sec:Cp}

Consider a shrinking curve $\gamma:I\to\R^2$ parametrized in
arclength $s$, where $I\subset\R$ is an interval. We denote with
$R:\R^2\to\R^2$ the counterclockwise rotation of $90$ degrees. Then, the relation
$$
\gamma_{ss}=\frac{d^2\gamma}{ds^2}=k=-\langle\gamma\,\vert\,\nu\rangle=-\Bigl\langle\gamma\,\Bigl\vert\,R\Bigl(\frac{d\gamma}{ds}\Bigr)\Bigr\rangle
$$ 
gives an ODE satisfied by $\gamma$. It follows that the curve is smooth
and it is not difficult to see that for every point $x_0\in\R^2$ and
unit velocity vector $\tau_0$, there exists a unique shrinking curve
(solution of such ODE) parametrized in arclength, passing at $s=0$
through the point $x_0$ with velocity $\tau_0$, defined for all $s\in\R$.

Differentiating in arclength the equation
$k=-\langle\gamma\,\vert\,\nu\rangle$, 
we get the ODE for the curvature $k_s=k\langle
\gamma\,\vert\,\tau\rangle$. Suppose that at
some point $k=0$, then it must also hold $k_s=0$ at the same point,
hence, by the uniqueness theorem for ODEs we conclude that $k$ is
identically zero and we are dealing with a line $L$ which, as $\langle
x\,\vert\,\nu\rangle=0$ for every $x\in L$, must contain the origin
of $\R^2$.

So we suppose that $k$ is always nonzero and, by looking at the
structural equation $k+\langle\gamma\,\vert\,\nu\rangle=0$, 
we can see that the curve is then strictly
convex with respect to the origin of $\R^2$. Another consequence (by
the uniqueness theorem for ODE) is that the curve must be symmetric
with respect to any critical point (maximum or minimum) of its curvature
function: Notice that if the curve is not a piece of a circle, they 
are all nondegenerate and isolated (if the curve has
bounded length, their number is finite).

Computing the derivative of $\vert\gamma\vert^2$,
\begin{equation*}
\frac{d\vert\gamma\vert^2}{ds}=2\langle\gamma\,\vert\,\tau\rangle=2k_s/k=2\frac{d\log{k}}{ds}\,
\end{equation*}
we get $k=Ce^{\vert \gamma\vert^2/2}$ for some constant $C\in\R$, that
is, the quantity 
\begin{equation}\label{energyeq}
{\mathcal{E}}={\mathcal{E}}(\gamma):=ke^{-\vert \gamma\vert^2/2},
\end{equation}
that we call {\em Energy}, is constant along the curve. 
Equivalently,  $\langle\gamma\,\vert\,\nu\rangle e^{-\vert
  \gamma\vert^2/2}$ is constant. A solution $\g$ has positive energy
if $k > 0$, so that $\g$ runs counterclockwise around the origin,  
$\g$ has negative energy if $k<0$, so that $\g$ runs clockwise around
the origin, $\g$ has energy zero if $k = 0$, so that $\g$ is a piece
of a straight line through the origin.

We consider now a new coordinate $\theta=\arccos{\langle
  e_1\,\vert\,\nu\rangle}$; this can be done for the whole curve 
as we know that it is convex (obviously, $\theta$ is only locally
continuous, since it ``jumps'' after a complete round).

Differentiating with respect to the arclength parameter we have $\frac{d\theta}{ds}=k$ and
\begin{equation}\label{eqqq}
k_\theta=k_s/k=\langle\gamma\,\vert\,\tau\rangle\qquad
k_{\theta\theta}=\frac{1}{k}\frac{dk_\theta}{ds}
=\frac{1+k\langle\gamma\,\vert\,\nu\rangle}{k}
=\frac{1}{k}-k.
\end{equation}
Multiplying both sides of the last equation by $2k_\theta$ we get
$\frac{d\,}{d\theta}[k_\theta^2+k^2-\log{k^2}]=0$, that is,
the quantity
\begin{equation*}
E:=k_\theta^2+k^2-\log{k^2}
\end{equation*}
is constant along all the curve. Notice that such quantity $E$ cannot be less than
1 (if $k\not=0$), moreover, if $E=1$ we have that $k^2$ must be constant and equal to one
along the curve, which consequently must be a piece of the unit circle centered
at the origin of $\R^2$.

As $E\geq 1$, it follows that $k^2$ is uniformly bounded from above
and away from zero, hence, recalling that $k={\mathcal{E}}e^{\vert \gamma\vert^2/2}$, the curve
$\gamma$ is contained in a ball of $\R^2$ (and it is outside some
small ball around the origin).

Since we are interested in the curves of a non--trivial connected,
compact ($\Theta$--shaped), regular network, there will be no unbounded lines or
complete circles and all the curves of the network will be images of a
closed bounded interval, once parametrized in arclength.

Resuming, either $\gamma$ is a segment or 
$k^2>0$, the equations~\eqref{eqqq} hold, the Energy
${\mathcal{E}}=ke^{-\vert \gamma\vert^2/2}$ and the quantity
$E=k_\theta^2+k^2-\log{k^2}>1$ are constant along the curve, where $\theta=\arccos{\langle
  e_1\,\vert\,\nu\rangle}$. Moreover, the curve is locally symmetric with respect to the critical
points of the curvature, hence the curvature $k(\theta)$ is
oscillating between its maximum and its minimum.

Suppose now that $k_{\min} < k_{\max}$ are these two consecutive critical
values of $k$. It follows that they are two distinct positive zeroes
of the function $k_\theta^2=E+\log{k^2}-k^2$, when $E>1$, with $0 < k_{\min} < 1 <
k_{\max}$.\\
We have then that the change $\Delta\theta$ in the angle $\theta$ along
the piece of curve delimited by two consecutive points where the
curvature assumes the values $k_{\min}$ and $k_{\max}$, is given by the integral
\begin{equation} \label{Delta theta}
\Delta\theta = I(E)=\int_{k_{\min}}^{k_{\max}}\frac{dk}{\sqrt{E-k^2+\log{k}^2}}.
\end{equation}

\begin{prop}[Abresch and Langer~\cite{ablang1}]\label{prop:I} 
The function $I:(1,+\infty)\to\R$ satisfies
\begin{enumerate}
\item $\lim_{E\to 1^+} I(E)=\pi/\sqrt{2}$,
\item $\lim_{E\to +\infty} I(E)=\pi/2$,
\item $I(E)$ is monotone nonincreasing.
\end{enumerate}
As a consequence $I(E)>\pi/2$.
\end{prop}

We write now the curve $\g$ in polar coordinates, that is, $\g(s)=(\rho(s) \cos \th(s),\rho(s) \sin \th(s))$, then, 
the arclength constraint and the shrinker equation~\eqref{1102.1-bis} become
\begin{equation}\label{arclength}
\rho_s^2 + \rho^2 \th_s^2 = 1,
\end{equation}
\begin{equation*}
\rho^2 \th_s +\rho \rho_{ss} \th_s
- 2 \rho_s^2 \th_s - \rho^2 \th_s^3 - \rho \rho_s \th_{ss} = 0,
\end{equation*}
moreover, 
\begin{equation} \label{1602.1}
\cos \big( \text{angle between $\g$ and $\g_s$} \big) 
= \frac{\g \cdot \g_s}{|\g| |\g_s|} \, 
= \rho_s.
\end{equation}
Notice that shrinking curves with positive energy have $\th_s > 0$
everywhere, indeed, either $\th_s$ is always different by zero or the
curve is a segment of a straight line for the origin of $\R^2$.

The curvature and the Energy $\mE=ke^{-\vert \gamma\vert^2/2}$ are given by 
\begin{equation} \label{1702.2}
k = \rho^2 \th_s, \qquad 
\mE = \rho^2 \th_s e^{-\frac12 \rho^2}
\end{equation}
and, when the energy is positive, it will be useful to consider also the quantity $\mF:=- \log(\mE)$,
that is, 
\begin{equation}\label{energy}
\mF = - \log(\mE) = \frac12 \rho^2 - \log (\rho^2 \th_s).
\end{equation}
Since $0 < \rho\th_s \leq 1$, by equation~\eqref{arclength}, one has
\begin{equation*}
\mF \geq \frac12 \rho^2 - \log (\rho) \geq \frac12.
\end{equation*}

Let us assume that $\g$ is a shrinking curve with $k>0$ (the
assumption on the sign of $k$ is not restrictive, up to a change of
orientation of the curve). Then, by the definition of the
Energy~\eqref{energyeq}, it is immediate to see that the points where $k$ attains its maximum (resp. minimum)
coincide with the points where $\rho$ attains its maximum
(resp. minimum). Thus, at any extremal point of $k$ there hold 
$k_\theta=0$, $\rho_s = 0$ and also $\rho \th_s=1$, by
equation~\eqref{arclength}, hence, by equation~\eqref{1702.2}, we have
$k=\rho$. Then, computing $E$ and $\mF$ at such point
(clearly, $k_\theta=0$), we get
$$
E=k^2-2\log k\qquad\qquad\text{ and }\qquad\qquad \mF = k^2/2 - \log k,
$$
that is, $E=2\mF= \log \bigl( \frac{1}{\mE^2} \bigr)$.

Since the Energy and the quantity $\mF$ are constant, 
this relation must hold along all the curve $\gamma$ and ${\mathcal{F}}=\rho_{\rm min}^2/2-\log\rho_{\rm min}=\rho_{\rm max}^2/2-\log\rho_{\rm max}$.

Since the function $\mu(t)=t^2/2-\log t$ is strictly convex with a minimum value $1/2$ at $t=1$, to each value of $\mF\geq\frac12$, there correspond
two values $\rho_{\rm min}(\mF)$ and $\rho_{\rm max}(\mF)$ which are
the admissible (interior) minimum and maximum of $\rho$ on $\gamma$, 
with $\rho_{\rm min}(\mF)<1<\rho_{\rm max}(\mF)$ if $\mF>\frac12$. It 
follows easily that $\rho_{\rm  max}:(1/2,+\infty)\to(1,+\infty)$ is an increasing function and
$\rho_{\rm min}:(1/2,+\infty)\to(0,1)$ is a decreasing function. Viceversa, the quantity ${\mathcal{F}}$ can be seen as a decreasing function of $\rho_{\rm min}\in(0,1]$ and an increasing function of $\rho_{\rm max}\in[1,+\infty)$.

Let $s_{\rm min}, s_{\rm max} \in \R$ 
with $s_{\rm min} < s_{\rm max}$ be two consecutive (interior) extremal points of
$\rho$ (hence, also of $k$) such that 
$\rho(s_{\rm min}) = \rho_{\rm min}(\mF)$, 
$\rho(s_{\rm max}) = \rho_{\rm max}(\mF)$.
Since at the interior extremal points of $\rho$ the vectors $\g,\g_s$ must be orthogonal, 
it follows that the quantity considered in formula~\eqref{Delta theta} satisfies
\begin{equation}\label{mImF}
\Delta\theta = \int_{s_{\rm min}}^{s_{\rm max}} \th_s(s) \, ds\,:=  \mI(\mF),	
\end{equation}
that is, the integral $\mI(\mF)$ is the variation of the angle $\th$ on the shortest arc such that $\rho$ passes from $\rho_{\rm min}$ to $\rho_{\rm max}$.\\
Then, by the above discussion, $\mI(\mF)=I(E)=I(2\mF)$ and we can
rephrase Proposition~\ref{prop:I} in terms of the integral $\mI(\mF)$ as follows.

\begin{prop}\label{prop:mI}
The function $\mI:(1/2,+\infty)\to\R$ satisfies
\begin{enumerate}
\item $\lim_{\mF\to (1/2)^+} \mI(\mF) = \frac{\pi}{\sqrt{2}}$,
\item $\lim_{\mF\to +\infty} \mI(\mF)= \frac{\pi}{2}$,
\item $\mI(\mF)$ is monotone nonincreasing.
\end{enumerate}
As a consequence $\mI(\mF) > \frac{\pi}{2}$ for all $\mF > \frac12$.
\end{prop}

\section{The proof of Theorem~\ref{thm:main}}
\label{sec:pof}

The proof of Theorem~\ref{thm:main} is based on the following lemma whose proof is postponed to Section~\ref{sec:proof of lemma zucca}.

\begin{lemma} 
\label{lemma:zucca}
Let $\g$ be a shrinking curve, parametrized counterclockwise by arclength, with positive curvature and let $(s_0, s_1)$ be an interval where $s \mapsto \rho(s)$ is increasing. 
If $\rho_s(s_0) \geq \frac12$, 
namely, if the angle formed by the vectors $\g(s_0)$ and $\g_s(s_0)$ is $\leq \frac{\pi}{3}$, 
then 
\begin{equation} \label{1503.3}
\int_{s_0}^{s_1} \th_s(s) \, ds < \frac{\pi}{2}.
\end{equation}
Similarly, if $s \mapsto \rho(s)$ is decreasing on $(s_0, s_1)$ and $\rho_s(s_1) \leq - \frac12$, 
namely the angle formed by the vectors $\g(s_1)$ and $\g_s(s_1)$ is $\geq \frac{2\pi}{3}$, 
then the same conclusion holds.
\end{lemma}

\begin{remark}
Proving estimate~\eqref{1503.3} is equivalent to show that
\begin{equation} \label{1503.1}
\int_{s_0}^{s_1} \theta_s(s) \, ds < \frac{2 \pi}{3} ,
\end{equation}
where $\theta(s)$ is the angle formed by $e_1 = (1,0)$ and the normal vector $\nu(s)$. Indeed, clearly 
\[
\int_{s_0}^{s_1} \theta_s(s) \, ds \leq \int_{\s_0}^{\s_1} \theta_s(s) \, ds,
\quad 
\int_{s_0}^{s_1} \th_s(s) \, ds \leq \int_{\s_0}^{\s_1} \th_s(s) \, ds,
\]
where $k(\s_1) = k_{\text{max}}$, 
and $\s_0$ is the maximum $\s\leq s_0$, assuming it exists, such that the angle formed by the vectors 
$\g(\s)$ and $\g_s(\s)$ equals $\frac{\pi}{3}$ 
and the map $s \mapsto \rho(s)$ is increasing on $(\s, \s_1)$. 
Then one observes (by elementary angle geometry) that
\[
\int_{\s_0}^{\s_1} \theta_s(s) \, ds
= \int_{\s_0}^{\s_1} \th_s(s) \, ds + \frac{\pi}{6} .
\]
The integral in~\eqref{1503.1} can be expressed as before 
\begin{equation*}
\int_{s_0}^{s_1} \theta_s(s) \, ds 
= \int_{k(s_0)}^{k(s_1)} \frac{dk}{\sqrt{E-k^2+\log{k}^2}},
\end{equation*}
hence it is bounded by $I(E)$, defined in formula~\eqref{Delta theta} 
(because, in general, $k_{\text{min}} \leq k(s_0) \leq k(s_1) \leq k_{\text{max}}$). 
We know that $I(E) < \frac{\pi}{\sqrt 2}$, but being $\frac{2 \pi}{3} < \frac{\pi}{\sqrt 2}$, estimate~\eqref{1503.1} is not a direct consequence of Proposition~\ref{prop:I}.

Even if such integral is well studied, we found it easier to prove
estimate~\eqref{1503.3} than to show that
\begin{equation*}
\int_{s_0}^{s_1} \theta_s(s) \, ds 
= \int_{k(s_0)}^{k(s_1)} \frac{dk}{\sqrt{E-k^2+\log{k}^2}}<\frac{2\pi}{3}
\end{equation*}
and this is the reason for our introduction and computation in polar coordinates $(\rho,\phi)$.
\end{remark}

\medskip

We assume now that a $\Theta$--shrinker exists, described by three embedded shrinking curves $\g_i:[\underline{s}_i,\overline{s}_i]\to\R^2$, parametrized by arclength, expressed in polar coordinates by $\g_i = (\rho_i \cos(\th_i), \rho_i \sin(\th_i))$, for $i\in \{1,2,3\}$. The two triple junctions will be denoted with $A,B$ and the three curves intersect each other only at $A$ and $B$ (which are their endpoints) forming angles of $120$ degrees. Since the shrinker equation~\eqref{1102.1} is invariant by rotation, we can assume that the segment $\overline{AB}$ is contained in the straight line 
$\{ (x,q) : x \in \R \}$ with $q \geq 0$ and we let $A = (x_A,q)$, $B = (x_B,q)$ with $x_A < x_B$.

We begin with some preliminary elementary lemmas. To simplify the notation, in all this section we will denote the arclength derivative $\frac{d\,}{ds}$ with $'$.

\begin{lemma} \label{lemma:no giro}
For all $i\in\{1,2,3\}$, the curve $\g_i$ is either a straight line or such that 
$$
\biggl| \int^{\overline{s}_i}_{\underline{s}_i} \th_i'(s) \, ds \biggr| \, < 2\pi.
$$
\end{lemma} 

\begin{proof}
Without loss of generality, assume that all $\g_1, \g_2, \g_3$ start at $B$ and end at $A$, namely 
$\g_i(\underline{s}_i) = B$, $\g_i(\overline{s}_i) = A$, for $i\in\{1,2,3\}$.

Assume, by contradiction, that $\g_1$ is a curve with positive energy and curvature such that 
\[
\int_{\underline{s}_1}^{\overline{s}_1} \th_1'(s) \, ds \geq 2\pi.
\]
Then there exist $\s_1, \t_1 \in S_1$ such that 
\[
\int_{\underline{s}_1}^{\s_1} \th_1'(s) \, ds = 2\pi, 
\quad 
\int_{\t_1}^{\overline{s}_1} \th_1'(s) \, ds = 2\pi.
\]
Since $\g_1$ does not intersect itself, 
one has $(\rho_1(\s_1) - \rho_1(\underline{s}_1)) (\rho_1(\overline{s}_1) - \rho_1(\t_1)) > 0$. 
Assume, without loss of generality, that 
\[
\rho_1(\s_1) < \rho_1(\underline{s}_1), \quad 
\rho_1(\t_1) > \rho_1(\overline{s}_1).
\] 
Now consider the triple junction at the point $B$, the straight line $r$ passing 
through $B$ and the origin, and let $H_1$ and $H_2$ be the open half-planes in which $r$ divides $\R^2$, where $H_1$ is the one containing $\g_1'(\underline{s}_1)$. 
Since the three curves $\g_1$, $\g_2$, $\g_3$ form angles of $\frac{2\pi}{3}$ at $B$, 
at least one among $\g_2'(\underline{s}_2)$ and $\g_3'(\underline{s}_3)$ belongs to $H_2$. 
Without loss of generality, let $\g_2'(\underline{s}_2) \in H_2$. 
Since $\th_2'$ never vanishes, the curve $\g_2$ cannot reach the endpoint $A$ 
without crossing the curve $\g_1$ at some interior point, which is a contradiction.
\end{proof}

\begin{lemma} \label{lemma:ut}
Let $S= [\underline{s},\overline{s}]$ and $\g:S\to\R^2$ be a shrinking curve parametrized by arclength, expressed in polar coordinates by $\g = (\rho \cos(\th), \rho\sin(\th))$. Assume that $\th'(s) > 0$ in $S$ and 
\[
0 < \Delta \leq \pi, \quad \text{where} \ \ 
\Delta := \int_{\underline{s}}^{\overline{s}} \th'(s) \, ds.
\]
Let $L$ be the straight line passing through the two points $\g(\underline{s})$, $\g(\overline{s})$ and $H_1$ and $H_2$ be the two closed half--planes in which $L$ divides the plane $\R^2$. 
Then the arc $\g(S)$ is entirely contained in $H_1$ or $H_2$.

Moreover, if $\Delta < \pi$ and $\g(S)\subset H_1$, then the origin of $\R^2$ belongs to the interior of $H_2$. 
\end{lemma}

\begin{proof}
By the assumption $\th'>0$, we have $k>0$ and the arc $\g(S)$ is contained in the cone $\mathcal C:=\{ \th(\underline{s}) \leq \th \leq \th(\overline{s}) \}$, which is convex by the assumption $0 < \Delta \leq \pi$. Since the curvature is positive, the closed set $\mathcal T$ delimited by the arc $\g(S)$ and by the two line segments joining the origin with $\g(\underline{s})$ and $\g(\overline{s})$ is a convex subset of $\R^2$, hence, the line segment joining $\g(\underline{s})$ with $\g(\overline{s})$ is contained in $\mathcal T$, which implies the thesis.
\end{proof}

Coming back to our $\Theta$--shrinker, because of its topological structure, one of the curves is contained in the region delimited by the other two, moreover the curvature of both these two ``external'' curves is always non zero, otherwise any such curve is a segment of a straight line passing for the origin, then the $120$ degrees condition at its endpoints would imply that it must be contained in the region bounded by the other two curves, hence it could not be ``external''. Notice that, on the contrary, the ``inner'' curve could actually  be a segment for the origin.

We call $\g_2$ the ``inner'' curve and, recalling that the origin of $\R^2$ is not over the straight line through the two triple junctions $A$ and $B$, parametrizing counterclockwise the three curves, that is $\th'_i>0$ (in the case that the ``inner'' curve $\g_2$ is not a segment), we call $\g_1$ the ``external'' curve which starts at $B$. By Lemma~\ref{lemma:no giro}, $\g_1$ reaches the point $A$ after $\th_1$ changes of an angle $\Delta= \int_{\underline{s}_1}^{\overline{s}_1} \th_1'(s) \, ds<2\pi$ equal to the angle $\widehat{BOA}$, which is smaller or equal than $\pi$. Hence, by Lemma~\ref{lemma:ut}, all such curve $\g_1$ stays over the straight line passing for  the two triple junctions $A$ and $B$.

We call $\g_3$ the other extremal curve, hence since $\th_1,\th_3>0$, we have 
\[
\g_1(\underline{s}_1) = \g_3(\overline{s}_3) = B, \quad 
\g_1(\overline{s}_1) = \g_3(\underline{s}_3) = A. 
\]

Because of the shrinker equation~\eqref{1102.1-bis}, all the three curves are convex with respect to the origin. This implies that the origin is contained in the interior of the bounded area $A_{13}$ enclosed by 
$\g_1$ and $\g_3$ (if the origin belongs to $\g_1$ or $\g_3$ such curve is a segment and cannot be ``external'', as we said before), which also contains $\g_2 \subset A_{13}$. We let $A_{12}$ be the region enclosed by the curves $\g_1$ and $\g_2$ and we split the analysis into two cases.

\bigskip

\emph{Case 1. The origin does not belong to the interior of $A_{12}$.}

\smallskip

Since the curve $\g_2$ is convex with respect to the origin, by the
same argument used above for $\g_1$, it is contained in the upper
half--plane determined by the straight line for the points $A$ and
$B$.

By the $120$ degrees condition it follows that the angle $\b$ at $B$ formed by the vector $(1,0)$ and $\g_1'$ is at most $\frac{\pi}{3}$. 
Similarly, also the angle $\a$ at $A$ formed by the vector $(1,0)$ and $\g_1'$ is at most $\frac{\pi}{3}$.
By the convexity of the region delimited by $\g_2$ and $\g_3$ containing the origin and again the $120$ degrees condition at $B$, 
it is then easy to see that the angle at $B$ formed by the vectors $\g_1$ and $\g_1'$ is less or equal than $\frac{\pi}{3}$ and analogously, the angle at $A$ formed by 
$\g_1$ and $\g_1'$ is greater or equal than $\frac{2\pi}{3}$.

Hence, by equality~\eqref{1602.1}, it follows
\[
\rho_1'( \underline{s}_1 ) \geq \frac12 > 0, \quad 
\rho_1'( \overline{s}_1 ) \leq - \frac12 < 0.
\]
As a consequence, there is a point of maximum radius $s_1^* \in (\underline{s}_1, \overline{s}_1 )$ 
such that $\rho_1(s_1^*) \geq \rho_1(s)$ for all $s\in
(\underline{s}_1, \overline{s}_1)$.

The vector $\g_1(s_1^*)$ forms an angle $\sigma\geq \frac{\pi}{2}$
with $(1,0)$ or $(-1,0)$. Assume that the angle between $\g_1(s_1^*)$
and $(1,0)$ is greater or equal than $\frac{\pi}{2}$ (the other case
is analogous, switching $A$ and $B$). We extend the curve $\g_1$
(still parametrized by arclength) ``before'' the point $B$ till it
intersects the $x$--axis at some $\widetilde s_1 \leq \underline{s}_1$
(this must happen because $\th_1(s)>0$ everywhere also on the extended
curve) and we consider the (non relabeled) curve $\g_1$ defined in the
interval $L_1= [\widetilde s_1, s_1^*]$. Calling $\b_0$ the angle
formed by the vectors $\g_1'(\widetilde s_1)$ and $(1,0)$, by
convexity and the fact that the angle $\b$ at $B$ formed by the vector
$(1,0)$ and $\g_1'$ is at most $\frac{\pi}{3}$, we have that $\b_0
\leq \b \leq \frac{\pi}{3}$. 
Hence, by equality~\eqref{1602.1}, we have $\rho_1'(\widetilde{s}_1)
\geq \frac12 > 0$.

Considering now the function $s \mapsto \rho_1(s)$ on the interval
$L_1= [\widetilde s_1, s_1^*]$, since $\rho_1'(\widetilde{s}_1) > 0$
and $s_1^*$ is a maximum point for $\rho_1$, either $\rho_1$ is
increasing on $L_1$, or $\rho_1$ has another maximum and then a
minimum in the interior of $L_1$ (notice that the map $\rho_1$ cannot
be constant on an interval, otherwise $\g_1$ would be an arc of a
circle centered at the origin, which is impossible since $\rho_1$ is not constant). But we know from formula~\eqref{Delta theta}
and Proposition~\ref{prop:I} that the angle $\th_1$ must increase more
than $\frac{\pi}{2}$ to go from a minimum to a maximum or viceversa
(we can apply such proposition since $\g_1$ is not an arc of a
circle). Since 
\[
\int_{\widetilde s_1}^{s^*_1} \th_1'(s) \, ds \leq \pi,
\]
there cannot be a maximum, then a minimum, then a second maximum in $L_1$. 
It follows that $\rho_1$ is increasing in such interval.\\
This, combined with the fact that $\b_0 \leq \frac{\pi}{3}$ and that the angle $\sigma$ is at least $\frac{\pi}{2}$, that is, $\int_{\widetilde s_1}^{s_1^*} \th_1'(s)\,ds \geq \frac{\pi}{2}$, is in contradiction with Lemma~\ref{lemma:zucca}. Therefore, this case cannot happen.

\bigskip

\emph{Case 2. The origin belongs to the interior of $A_{12}$.}

\smallskip
  
Being the region $A_{12}$ convex (by the shrinker equation~\eqref{1102.1-bis}, since it contains the origin), the curve $\g_2$ (which is oriented counterclockwise) goes from $A$ to $B$.
The fact that $\g_2'$ and $\g_3'$ form angles of $\frac{2\pi}{3}$ at the points $A$ and $B$ implies that: 

$(i)$ the angle in $A$ formed by the vectors $\g_3(\underline{s}_3)$ and $\g_3'(\underline{s}_3)$ 
and the angle in $B$ formed by the vectors $\g_2(\overline{s}_2)$ and $\g_2'(\overline{s}_2)$ are both less or equal than $\frac{\pi}{3}$; 

$(ii)$
the angle in $B$ formed by the vectors $\g_3(\overline{s}_3)$ and $\g_3'(\overline{s}_3)$ 
and the angle in $A$ formed by the vectors $\g_2(\underline{s}_2)$ and $\g_2'(\underline{s}_2)$ 
are both greater or equal than $\frac{2\pi}{3}$. 

In particular, by equality~\eqref{1602.1}, it follows 
\begin{equation} \label{1902.1}
\rho_2'(\underline{s}_2) \leq -\frac12 < 0, \quad 
\rho_2'(\overline{s}_2) \geq \frac12 > 0, \quad 
\rho_3'(\underline{s}_3) \geq \frac12 > 0, \quad 
\rho_3'(\overline{s}_3) \leq - \frac12 < 0. 
\end{equation}
Hence, the function $s \mapsto \rho_3(s)$ has a maximum 
at some point $s^*_3\in(\underline{s}_3, \overline{s}_3 )$, 
while the function $s \mapsto \rho_2(s)$ has a minimum 
at some point $s^\circ_2 \in (\underline{s}_2, \overline{s}_2)$. 

If $s^*_3$ is the only point of maximum of $\rho_3$ in the interval $[\underline{s}_3, \overline{s}_3 ]$, then the function $\rho_3$ is strictly monotone on each of the two subintervals $[\underline{s}_3, s^*_3 ]$ and $[s^*_3, \overline{s}_3 ]$, moreover,
 \[
\int_{\underline{s}_3}^{s^*_3}\th_3'(s)\, ds +  \int_{s^*_3}^{\overline{s}_3}\th_3'(s)\, ds = \int_{\underline{s}_3}^{\overline{s}_3}\th_3'(s)\, ds \geq \pi,
\]
since the origin is ``below'' the segment $\overline{AB}$. Thus, at least one of the two integrals on the left--hand side is greater or equal than $\frac{\pi}{2}$ 
and, by Lemma~\ref{lemma:zucca}, this is not possible. 
As a consequence, there must be another point of maximum radius
$s^{**}_3 \in (\underline{s}_3, \overline{s}_3 )$ (notice that the
maximum points cannot be an interval, otherwise $\g_3$ would be an arc
of a circle centered at the origin, hence with $\rho_3'=0$, against
relations~\eqref{1902.1}). Hence, between these two points of maximum
radius there is a minimum point $s^\circ_3$. Without loss of
generality, we assume that  $\underline{s}_3 < s^*_3 < s^\circ_3 <
s^{**}_3 < \overline{s}_3$. 

We observe that there cannot be a third maximum point for $\rho_3$
(hence also another minimum point) in the interval $[\underline{s}_3,
\overline{s}_3 ]$ because, by Lemma~\ref{prop:mI}, each of the four angles at the origin formed
by the segment connecting the origin with two consecutive of the five extremal
points for $\rho_3$ on $\g_3$ is greater than $\frac{\pi}{2}$ and, by
Lemma~\ref{lemma:no giro}, there holds
$\int_{\underline{s}_3}^{\overline{s}_3} \th_3'(s) \, ds <
2\pi$. Moreover, also the case of two minimum points and two maximum
points for $\rho_3$ in the interval $[\underline{s}_3, \overline{s}_3
]$ is not possible, because of the sign of the derivative $\rho_3'$ at the endpoints in
relations~\eqref{1902.1}. Hence, we conclude that $s^*_3, s^\circ_3,
s^{**}_3$ are the only extremal points for $\rho_3$ in the interval
$[\underline{s}_3, \overline{s}_3 ]$.

Now consider the quantities $\mF_2,\mF_3$ of the curves $\g_2,\g_3$, respectively, given by formula~\eqref{energy}. 
By relations~\eqref{1902.1}, the curves $\g_2$ and $\g_3$ are not the
unit circle (they would have $\rho_2'$ or $\rho_3'$ equal to zero everywhere), therefore $\mF_2, \mF_3 > \frac12$.
If we draw the line from the origin to $\g_3(s^\circ_3)$, this must
intersect $\g_2$ in an intermediate point, implying that the minimal
radius of the curve $\g_2$ is smaller than the minimal radius of the
curve $\g_3$. By the discussion about the value of the quantity $\mF$ in relation with the extremal values of $\rho$ at the end of Section~\ref{sec:Cp}, we have $\mF_2 >
\mF_3$. Then, if a maximum of $\rho_2$ is taken in the interior of $\g_2$, it must be larger than the maximal radius of $\g_3$ (which is taken in the interior of $\g_3$), which is not possible as $\g_2$ is contained in the region bounded by $\g_3$ and the segment $\overline{AB}$. From
this argument we conclude that there are no points of maximal radius in the interior of $\g_2$, thus, the only extremal point for $\rho_2$ in the interval $[{\underline{s}_2},{\overline{s}_2}]$ is the minimum point $s^\circ_2$. 

Defining the angle
\[
\alpha:= \int_{\underline{s}_2}^{\overline{s}_2}\th_2'(s)\, ds = \int_{\underline{s}_3}^{\overline{s}_3}\th_3'(s)\, ds,
\]
by formula~\eqref{mImF} and the symmetry of the curve $\g_3$ with respect to the straight line through the origin and the point $\g_3(s_3^\circ)$ of minimum distance, we have
\[
\mI(\mF_3) = \int_{s^*_3}^{s^\circ_3}\th_3'(s)\, ds 
= \int_{s^\circ_3}^{s^{**}_3}\th_3'(s)\, ds 
< \frac{\alpha}{2}
\]
while, since $\g_2$ does not contain any interior point of maximum radius,
\[
\mI(\mF_2) > \max \left\{ \int_{\underline{s}_2}^{s^\circ_2}\th_2'(s)\, ds, \; \int_{s^\circ_2}^{\overline{s}_2}\th_2'(s)\, ds \right\} 
\geq \frac{\alpha}{2}.
\]
Thus, $\mI(\mF_2) > \mI(\mF_3)$ and $\mF_2 > \mF_3$, which is in
contradiction with the monotonicity of the function ${\mathcal{I}}$
given by Proposition~\ref{prop:mI}. Hence, also this case can be
excluded.

\bigskip

Since we excluded both cases, our hypothetical $\Theta$--shrinker
cannot exist (under the validity of Lemma~\ref{lemma:zucca} which we
are going to prove in the next section).

\section{The proof of Lemma~\ref{lemma:zucca}}
\label{sec:proof of lemma zucca}

We prove Lemma~\ref{lemma:zucca} in the case of $\rho$ increasing, for $\rho$ decreasing the proof is similar and it can be deduced by changing the orientation of the curve.

Let $\g = (\rho \cos\th, \rho \sin \th)$ be a counterclockwise arclength parametrized shrinking curve, that we suppose defined for every $s\in\R$, and let $[s_0 , s_1]$ be an interval where $\rho(s)$ is increasing. 
As a first step, we prove a formula for the integral in Lemma~\ref{lemma:zucca}, using $\rho$ as a new integration variable. Let $\mu:\R^+\to\R$ be the function  
\begin{equation} \label{0402.1}
\mu(x) := \frac12\, x^2 - \log(x),
\end{equation}
so that relation~\eqref{energy} gives 
\begin{equation} \label{1902.3}
\mF = \mu(\rho(s)) - \log(\rho(s) \th_s(s)) \quad \forall s \in \R.
\end{equation}
Let $R$ be the maximum value of $\rho$ along the whole $\g$, namely $R = \rho_{\rm max}(\mF)$ in the notation at the end of Section~\ref{sec:Cp}. We consider $R > 1$, namely we exclude the case when $\g$ is the unit circle
(for the unit circle Lemma~\ref{lemma:zucca} trivially holds, as $\rho_s=0$ everywhere).
Since $\rho_s = 0$ when $\rho = R$, 
by equation~\eqref{arclength} one has $\rho \th_s = 1$ at that point, therefore, by relation~\eqref{1902.3},
we deduce 
\begin{equation} \label{1902.4}
\mF = \mu(R),
\end{equation}
which gives a link between the quantity $\mF$ and the maximum distance $R$ of the curve $\g$.
Notice that, for the same reason, $\mF = \mu(r)$ where $r = \rho_{\rm min}(\mF)$. 

By equations~\eqref{1902.3} and~\eqref{1902.4}, it follows 
\begin{equation*}
\rho \th_s = \exp(\mu(\rho) - \mu(R)),
\end{equation*}
hence,
$$
\th_s = \frac{1}{\rho} \, e^{\mu(\rho) - \mu(R)},
$$
then, since $\rho_s > 0$ in $[s_0, s_1]$, by equation~\eqref{arclength}, we conclude 
\begin{equation} \label{1902.6}
\rho_s = \sqrt{ 1 - e^{2[\mu(\rho) - \mu(R)]} }.
\end{equation}
Therefore, changing the integration variable as $\rho= \rho(s)$, we get
\begin{equation} \label{1902.7}
\int_{s_0}^{s_1} \th_s(s) \, ds = \int_{\rho(s_0)}^{\rho(s_1)} f(\rho,R) \, d\rho
\end{equation}
where
\begin{equation} \label{1902.8}
f(\rho,R) := \frac{e^{\mu(\rho) - \mu(R)}}{\rho \sqrt{ 1 - e^{2[\mu(\rho) - \mu(R)]} } } .
\end{equation}
The function $f(\rho,R)$ is defined for $R > 1$ and $\rho \in (r,R)$, where $r = \rho_{\min}(\mu(R))\in(0,1)$. 

\medskip

Let $s_{\min} < s_{\max}$ be two points such that 
\[
\rho_s(s_{\min}) = 0, \quad 
\rho_s(s_{\max}) = 0, \quad 
\rho_s(s) > 0 \quad \forall s \in (s_{\min}, s_{\max}),
\]
so that $\rho(s_{\min}) = r = \rho_{\min}(\mF)$ and 
$\rho(s_{\max}) = R = \rho_{\max}(\mF)$. 
By equation~\eqref{arclength}, $\rho_s(s)$ is increasing when $\rho(s) \th_s(s)$ is decreasing; being the quantity $\mF$ constant along the curve $\g$, by equation~\eqref{1902.3}, this happens when $\mu(\rho(s))$ is decreasing, 
namely when $\mu'(\rho(s)) < 0$, that is, for $\rho(s) < 1$ (the function $\mu$ is strictly convex with a minimum at $x=1$). 
Therefore, $\rho_s$ is increasing on $[s_{\min}, \s]$ and decreasing on $[\s, s_{\max}]$, 
where $\s \in (s_{\min}, s_{\max})$ is the only point such that $\rho(\s) = 1$. 

We analyze now the points where $\rho_s \geq \frac12$. 
First we observe that if the quantity $\mF$ is too low, then there are
no such points. Indeed, since $\rho_s$ has its maximum when $\rho =
1$, namely at $s = \s$, by equation~\eqref{1902.6}, we have 
\[
\rho_s(\s) = \sqrt{1 - e^{1 - 2 \mu(R)}}= \sqrt{1 - e^{1 - 2 \mF}},
\]
as $\mu(\rho(\sigma))=\mu(1)=1/2$.\\
It follows that $\rho_s(\s) \geq \frac12$ if and only if
$$
\mF\geq \frac12\, - \log \Big( \frac{\sqrt{3}}{2} \Big),
$$
that is, if and only if $R \geq \overline R$, where 
$\overline R$ is the unique real number such that 
\begin{equation} \label{0402.2}
\mu(\overline R) = \frac12\, - \log \Big( \frac{\sqrt{3}}{2} \Big), \quad \overline R > 1.
\end{equation}
As a consequence, Lemma~\ref{lemma:zucca} trivially holds for curves $\g$ such that $R < \overline R$
(because the assumption $\rho_s(s_0) \geq \frac12$ is simply not satisfied).

Let $R \geq \overline R$. Let $\widehat{s}$ be the smallest number in the interval $(s_{\min}, s_{\max})$ such that $\rho_s(\widehat{s}) \geq \frac12$, namely 
$\widehat{s}$ is the unique number in $(s_{\min}, \s]$ such that $\rho_s(\widehat{s}) = \frac12$, then, clearly 
\begin{equation} \label{1902.8bis}
\int_{s_0}^{s_1} \th_s(s) \, ds \leq \int_{\widehat{s}}^{s_{\max}} \th_s(s) \, ds.
\end{equation}
Letting $d(R)=\rho(\widehat{s})$, by equation~\eqref{1902.6}, $d(R)$ is the unique solution of 
\begin{equation} \label{0402.3}
\mu(d(R)) = \mu(R) + \log \Big( \frac{\sqrt{3}}{2} \Big), \quad 0 < d(R) \leq 1,
\end{equation}
for any $R \geq \overline R$. Notice that $d(\overline R) = 1$ and $0 < d(R) < 1$ for $R > \overline R$. 
Finally, by equations~\eqref{1902.7} and~\eqref{1902.8bis}, we have 
\begin{equation*}
\int_{s_0}^{s_1} \th_s(s) \, ds 
\leq \int_{\widehat{s}}^{s_{\max}} \th_s(s) \, ds=\int_{d(R)}^R f(x,R) \, dx =: J(R).
\end{equation*}
Lemma~\ref{lemma:zucca} is then proved once we get the following bound.
\begin{prop} \label{prop:int}
For all $R \geq \overline R$ there holds 
\begin{equation*}
J(R) < \frac{\pi}{2}.
\end{equation*}
\end{prop}

\subsection{Proof of Proposition~\ref{prop:int} -- Preliminaries}\  

It is easy to see by the definition of $d(R)$ in formula~\eqref{0402.3} that $R \mapsto d(R)$ is a strictly decreasing function on $[\overline R, +\infty)$ with $d(\overline R)=1$. We let
\begin{equation} \label{0402.5}
\psi(t) := \frac{t}{\sqrt{1 - t^2}} , \quad t \in (0,1),
\end{equation}
then, the function $f$, defined by formula~\eqref{1902.8}, can be expressed as
\begin{equation*}
f(\rho,R) = \frac{e^{\frac12(\rho^2 - R^2)} R}{\rho^2 \sqrt{ 1 - e^{\rho^2 - R^2} R^2 \rho^{-2} }} \,
= \frac{e^{\mu(\rho) - \mu(R)}}{\rho \sqrt{1 - e^{2[\mu(\rho) - \mu(R)]}}} \,
= \frac{1}{\rho} \, \psi\big( \exp\{\mu(\rho) - \mu(R)\} \big),
\end{equation*}
for $\rho\in [d(R), R)$. Indeed, $\mu(\rho) - \mu(R) < 0$ for all
$\rho\in [d(R),R)$.

Notice that, even if $f(\rho,R) \to + \infty$ as $\rho\to R^-$, the
integral $J(R)$ is finite for all $R \geq \overline R$, because $f$
diverges like $(R-\rho)^{-1/2}$ as $\rho \to R^-$.

\begin{lemma}[Approximation of $\overline R$] 
\label{lemma:approx bar R}
Recalling the defining formula~\eqref{0402.2} for $\overline R$, we have
\begin{equation} \label{0402.8} 
\sqrt{1 + \log(c_n)} < \overline R < \sqrt{1 + \log(d_n)}  \quad \forall n = 0,1,2,\ldots
\end{equation}
where $(c_n)$ is the increasing sequence defined by
\begin{equation} \label{0402.9} 
c_0 = 1, \quad c_{n+1} = \frac43 \, (1 + \log(c_n)), \quad n \geq 0
\end{equation}
and $(d_n)$ is the decreasing sequence defined by
\begin{equation} \label{2802.1} 
d_0 = 3, \quad d_{n+1} = \frac43 \, (1 + \log(d_n)), \quad n \geq 0.
\end{equation}
\end{lemma}

\begin{proof} 
By definition~\eqref{0402.2}, $\overline R > 1$ and the first inequality in~\eqref{0402.8} for $n=0$ holds.
Since $\frac54 > \log(3)$, one directly proves that 
$\mu(\sqrt{1 + \log(3)}) > \frac12 - \log(\frac{\sqrt{3}}{2}) = \mu(\overline R)$, 
where $\mu$ is defined by formula~\eqref{0402.1}. 
Since $\mu$ is strictly increasing on $(1,+\infty)$, it follows that $\overline R < \sqrt{1 + \log(3)}$, 
which is the second inequality in~\eqref{0402.8} for $n=0$. 

Assume that $\sqrt{1 + \log(c_n)} < \overline R < \sqrt{1 + \log(d_n)}$ for some $n \geq 0$. 
Then 
\[
1 + \log\Big( \frac43 \big( 1 + \log(c_n) \big) \Big) 
< 1 + \log\Big( \frac43 \overline R^2 \Big)
< 1 + \log\Big( \frac43 \big( 1 + \log(d_n) \big) \Big).
\]
By definition of $\overline R$ one has $1 + \log(\frac43 \overline R^2) = \overline R^2$, hence, by equations~\eqref{0402.9} and \eqref{2802.1}, we get 
$1 + \log(c_{n+1}) < \overline R^2 < 1 + \log(d_{n+1})$. Thus, formula~\eqref{0402.8} holds for all $n \geq 0$. 

For all $n \geq 1$ one has $c_{n+1} > c_n$ if and only if $c_n > c_{n-1}$ 
and similarly $d_{n+1} < d_n$ if and only if $d_n < d_{n-1}$.  
Since $c_1 > c_0$ and $d_1 < d_0$, 
the sequence $(c_n)$ is increasing and $(d_n)$ is decreasing.
\end{proof}

\begin{lemma}[Exact rounding for the value of $\overline R$]
\label{lemma:bar R value}
There holds $\frac{21}{16} < \frac75 < \overline R < \frac{23}{16}$.
\end{lemma}

\begin{proof} 
Using {\em interval arithmetic}, we compute exact roundings for $c_n$, $d_n$ 
in Lemma~\ref{lemma:approx bar R} for $n=50$, and we get 
$1.4004566266453120082 < \overline R < 1.4004566266453162271$.
\end{proof}

\begin{lemma}[Approximation of $d(R)$]
\label{lemma:approx d(R)}
Let $R \geq \overline R$, recalling the defining formula~\eqref{0402.3} for $d(R)$ and denoting 
\begin{equation} \label{0402.10} 
a := - \mu(R) - \log \Big( \frac{\sqrt{3}}{2} \Big),
\end{equation}
there holds 
\begin{equation} \label{0402.11} 
\exp(a_n) < d(R) \leq \exp(b_n), \quad \forall n = 0,1,2,\ldots
\end{equation}
where $(a_n)$ is the increasing sequence defined by
\begin{equation} \label{0402.12} 
a_0 = a, \quad a_{n+1} = a + \frac12\, \exp(2 a_n), \quad n \geq 0
\end{equation}
and $(b_n)$ is defined by
\begin{equation*}
b_0 = a + \frac12 , \quad b_{n+1} = a + \frac12\, \exp(2 b_n), \quad n \geq 0,
\end{equation*}
which is a decreasing sequence if $R > \overline R$. It is constant
equal to zero if $R=\overline{R}$.
\end{lemma}

\begin{proof}
Since $d(R) > 0$, from formulas~\eqref{0402.3} and~\eqref{0402.10} one has $\log(d(R)) = a + \frac12 d(R)^2 > a$, hence $d(R) > \exp(a_0)$. 
By induction, assume that $d(R) > \exp(a_n)$ for some $n \geq 0$. 
Then $\log(d(R)) = a + \frac12 d(R)^2 > a + \frac12 \exp(2 a_n) = a_{n+1}$, 
hence $d(R) > \exp(a_{n+1})$. This proves the first inequality~\eqref{0402.11} for all $n \geq 0$. 
Since $d(R) \leq 1$, one has $\log(d(R)) = a + \frac12 d(R)^2 \leq a + \frac12$, hence $d(R) \leq \exp(b_0)$. 
Assume that $d(R) \leq \exp(b_n)$ for some $n \geq 0$. Then $\log(d(R)) = a + \frac12 d(R)^2 
\leq a + \frac12 \exp(2 b_n) = b_{n+1}$, hence $d(R) \leq \exp(b_{n+1})$. The proof of inequalities~\eqref{0402.11} is complete. 
By the definition~\eqref{0402.12}, it follows that $a_{n+1} > a_n$ if and only if $a_n > a_{n-1}$, for all $n \geq 1$. Moreover, $a_1 > a_0$, therefore $(a_n)$ is increasing. 
Similarly, $b_{n+1} - b_n$ has the same sign of $b_n - b_{n-1}$ for all $n \geq 1$. 
For $R > \overline R$ one has $b_0 < 0$, hence $b_1 < b_0$ and the sequence $(b_n)$ is decreasing. 
Notice that for $R = \overline R$ one has $b_n = b_0 = 0$ for all $n$ and $d(\overline R) = 1 = \exp(b_n)$. 
\end{proof}

To approximate the integral $J(R)$ from above, 
we will use the convexity of the function $f(\cdot,R)$.

\begin{lemma}\label{lemma:convex f}
For all $R \geq \overline R$, the function $\rho \mapsto f(\rho,R)$ is strictly convex on $[d(R),R)$. 
\end{lemma}

\begin{proof} 
By definition, $f(\rho,R) = \frac{1}{\rho} \psi(\mG)$, where $\mG := \mG(\rho,R) = \exp(\mu(\rho) - \mu(R))$ and 
$\psi, \mu$ are defined in formulas~\eqref{0402.1} and~\eqref{0402.5}. We have $\mu'(\rho) = \rho - \rho^{-1}$, $\psi'(x) = (1-x^2)^{-3/2}$, 
and $\pa_\rho \mG = (\rho - \rho^{-1}) \mG$. Thus,
\begin{equation*}
\pa_\rho f(\rho,R) = \mG (1 - \mG^2)^{-\frac32} [1 + (\mG^2 - 2) \rho^{-2}], \quad 
\pa^2_{\rho \rho} f(\rho,R) = \frac{\mG}{(1 - \mG^2)^{\frac52} \rho^3} \, h
\end{equation*}
where
\begin{equation*}
h := \rho^4 (1 + 2\eta) - 3 \rho^2 (1+\eta) + (2\eta^2 -5\eta +6), \quad 
\eta := \mG^2.
\end{equation*}
Then, 
\begin{align*}
h = \Big( \rho^2 \sqrt{1+2\eta} - \frac{3(1+\eta)}{2\sqrt{1+2\eta}} \Big)^2 
+ \frac{1-\eta}{4(1+2\eta)}\, (15 + 25 \eta - 16 \eta^2),
\end{align*}
and $0 < \eta < 1$ for all $R \in [d(R),R)$ because $\mu(\rho) - \mu(R) < 0$. 
Hence, $1-\eta > 0$ and also $15 + 25 \eta - 16 \eta^2 > 0$, for all $\eta \in (0,1)$. 
Therefore, $h>0$ and $\pa_{\rho \rho} f(\rho,R) > 0$, for all $\rho \in [d(R),R)$, $R \geq \overline R$.
\end{proof}

\begin{lemma}\label{lemma:f decr in R}
For all $R > \overline R$, all $\rho \in [d(R),R)$, there holds  
$\partial_R f(\rho,R) < 0$, that is, the function $R \mapsto f(\rho,R)$ is decreasing.
\end{lemma}

\begin{proof} With the same notation as in the proof of Lemma~\ref{lemma:convex f}, we have 
$$
\pa_R f(\rho,R) = - \rho^{-1} \mG (1 - \mG^2)^{-3/2} (R - R^{-1}) < 0.
$$
\end{proof}

\subsection{Proof of Proposition~\ref{prop:int} for ``small'' $R$}

\begin{lemma}
There holds $J(R) < \frac32$ for all $R \in \bigl[\overline R, \frac{23}{16}\bigr]$. 
\end{lemma}

\begin{proof}
Recalling the bound in Lemma~\ref{lemma:bar R value}, we split 
\begin{equation*}
J(R) = \int_{d(R)}^1f(\rho,R) \, d\rho + \int_1^{\rho_2}f(\rho,R) \, d\rho + \int_{\rho_2}^{R}f(\rho,R) \, d\rho, 
\qquad \rho_2 := \frac{21}{16}.
\end{equation*}
Since $f(\cdot,R)$ is convex (Lemma~\ref{lemma:convex f}), 
the integral over the region $[d(R),1]$ is bounded by the area of the trapezoid
\begin{equation*}
\int_{d(R)}^1 f(\rho,R) \, d\rho \leq \frac{1 - d(R)}{2} \, \bigl( f(d(R),R) + f(1,R) \bigr).
\end{equation*}
Now $f(d(R),R) = \sqrt{3} / d(R)$, while, recalling that $d(\overline R) = 1$,  we have
\begin{align*}
f(1,R) =&\, \psi(\exp\{ \mu(1) - \mu(R) \})
\leq \psi(\exp\{ \mu(1) - \mu(\overline R) \})\\  
=&\, \psi(\exp\{ \mu(d(\overline R)) - \mu(\overline R) \})
= \psi \Big( \frac{\sqrt{3}}{2} \Big) = \sqrt{3}.
\end{align*}
Hence, using that $d(R) \geq d(23/16)$, we obtain
\begin{equation*}
\int_{d(R)}^1 f(\rho,R) \, d\rho 
\leq \frac{1 - d(R)}{2} \, \Big( \sqrt{3} + \frac{\sqrt{3}}{d(R)} \Big)
\leq \frac{1 - d(23/16)}{2} \, \Big( \sqrt{3} + \frac{\sqrt{3}}{d(23/16)} \Big).
\end{equation*}
The integral over the region $[1,\rho_2]$ is also bounded by the area of the corresponding trapezoid, namely,
\begin{equation} \label{0402.19}
\int_{1}^{\rho_2} f(\rho,R) \, d\rho 
< \frac{\rho_2 - 1}{2} \, \Big( \sqrt{3} + f(\rho_2,R) \Big).
\end{equation}
The integral over the region $[\rho_2,R)$ is bounded by using the change of variable 
\begin{equation} \label{cambio rho xi}
e^{\mu(\rho) - \mu(R)} = \xi, \quad 
\frac{(\rho^2 - 1)e^{\mu(\rho) - \mu(R)}}{\rho} \, d\rho = d \xi,
\end{equation}
which gives
\begin{align}
\int_{\rho_2}^R f(\rho,R) \, d\rho 
& = \int_{\rho_2}^R \frac{1}{\rho^2 - 1}\, \frac{1}{\sqrt{1 - e^{2[\mu(\rho) - \mu(R)]}}} \, 
\frac{(\rho^2 - 1) e^{\mu(\rho) - \mu(R)}}{\rho} \, d\rho 
\notag \\ & 
\leq  \frac{1}{\rho_2^2 - 1}\, \int_{\rho_2}^R \frac{1}{\sqrt{1 - e^{2[\mu(\rho) - \mu(R)]}}} \, 
\frac{(\rho^2 - 1) e^{\mu(\rho) - \mu(R)}}{\rho} \, d\rho 
\notag \\ & 
= \frac{1}{\rho_2^2 - 1}\, \int_{\xi(\rho_2)}^{\xi(R)} 
\frac{1}{\sqrt{1 - \xi^2}} \, d\xi
\notag \\ & 
= \frac{1}{\rho_2^2 - 1}\, \Big( \frac{\pi}{2}\, - \arcsin(e^{\mu(\rho_2) - \mu(R)}) \Big).
\label{0402.20}
\end{align}
Using Lemmas~\ref{lemma:bar R value} and~\ref{lemma:f decr in R}, 
we have $f(\rho_2,R) < f(\rho_2, 7/5)$ and $\mu(R) \leq \mu(23/16)$, 
and we insert these bounds in~\eqref{0402.19} and~\eqref{0402.20}. 
Since $\rho_2 = 21/16$, we get 
\begin{align} 
J(R) & 
\leq \frac{1 - d(23/16)}{2} \, \Big( \sqrt{3} + \frac{\sqrt{3}}{d(23/16)} \Big) 
+ \frac{5}{32} \, \Big( \sqrt{3} + f\Big(\frac{21}{16}, \frac75\Big) \Big) 
\notag \\ & \quad 
+ \frac{256}{185}\, \Big( \frac{\pi}{2}\, - \arcsin(e^{\mu(21/16) - \mu(23/16)}) \Big).
\label{0402.21}
\end{align}
Using Lemma~\ref{lemma:approx d(R)} and {\em interval arithmetic} for correct roundings, we obtain that the right--hand side of~\eqref{0402.21} is strictly less than 1.4712. 
\end{proof}

\subsection{Proof of Proposition~\ref{prop:int} for ``large'' $R$}

\begin{lemma}
There holds $J(R) < \frac32$ for all $R \in [4, +\infty)$. 
\end{lemma}

\begin{proof}
Let $R \geq 4$, we consider some $\rho_1, \rho_2$ such that $d(R) < \rho_1 < 1 < \rho_2 < R$ that we will choose later and we split the integral $J(R)$ of the four intervals.

\begin{enumerate}
\item The integral over $[d(R), \rho_1]$.\\ 
We use the change of variable~\eqref{cambio rho xi}. 
Since $\mu$ is decreasing on $(0,1)$, we get
\begin{equation} \label{0402.22}
\int_{d(R)}^{\rho_1} f(\rho,R) \, d\rho 
\leq \frac{1}{1 - \rho_1^2} \, \arcsin(\xi) \Big|_{\xi(\rho_1)}^{\xi(d(R))} 
= \frac{1}{1 - \rho_1^2} \, \Big( \frac{\pi}{3}\, - \arcsin(\xi(\rho_1)) \Big) 
< \frac{1}{1 - \rho_1^2} \, \frac{\pi}{3}.
\end{equation}

\item The integral over $[\rho_1,1]$.\\
Since $\mu$ is decreasing on $(0,1)$, we have
\begin{align}
\int_{\rho_1}^1 f(\rho,R) \, d\rho 
& = \int_{\rho_1}^1 \frac{1}{\rho} \, \psi \big( e^{\mu(\rho) - \mu(R)} \big) 
\notag \\ & 
\leq \psi \big( e^{\mu(\rho_1) - \mu(R)} \big) \int_{\rho_1}^1 \frac{1}{\rho} \, d\rho 
= \log \Big( \frac{1}{\rho_1} \Big) \psi \big( e^{\mu(\rho_1) - \mu(R)} \big).
\label{0402.23}
\end{align}

\item The integral over $[1,\rho_2]$.\\
Since $\mu$ is increasing on $(1,+\infty)$, we have
\begin{align}
\int_{1}^{\rho_2} f(\rho,R) \, d\rho 
& = \int_{1}^{\rho_2} \frac{1}{\rho} \, \psi \big( e^{\mu(\rho) - \mu(R)} \big) 
\notag \\ & 
\leq \psi \big( e^{\mu(\rho_2) - \mu(R)} \big) \int_1^{\rho_2} \frac{1}{\rho} \, d\rho 
= \log(\rho_2) \psi \big( e^{\mu(\rho_2) - \mu(R)} \big).
\label{0402.24}
\end{align}

\item The integral over $[\rho_2,R)$.\\
Using the general inequality
\begin{equation*}
\psi(e^x) = \frac{e^x}{\sqrt{1 - e^{2x}}} \, < \frac{1}{\sqrt{2|x|}} \quad \forall x < 0
\end{equation*}
we get 
\[
\psi(e^{\mu(\rho) - \mu(R)}) 
< \frac{1}{\sqrt{R^2 - \rho^2 - \log(R^2 \rho^{-2})}}
\]
and, since $\log(1+x) < x$ for all $x > -1$, we deduce that 
\[
f(\rho,R) < \frac{1}{\sqrt{(R^2 - \rho^2)(\rho^2 - 1)}} .
\]
Hence, 
\begin{equation} \label{0402.30}
\int_{\rho_2}^R f(\rho,R) \, d\rho 
< \frac{1}{\sqrt{(R + \rho_2)(\rho_2^2 - 1)}} \, \int_{\rho_2}^R \frac{1}{\sqrt{R-\rho}} \,d\rho
= \frac{2 \sqrt{R - \rho_2}}{\sqrt{(R + \rho_2)(\rho_2^2 - 1)}} .
\end{equation}
\end{enumerate}
The sum of the integrals~\eqref{0402.22},~\eqref{0402.23},~\eqref{0402.24} and~\eqref{0402.30}, 
with $\rho_2 = R-1$ and some $\rho_1$, gives 
\begin{equation} \label{0402.31}
J(R) \leq Q_1(\rho_1) + Q_2(\rho_1,R) + Q_3(R) + Q_4(R) =: Q(\rho_1,R)
\end{equation}
where
\begin{equation*}
Q_1(\rho_1) := \frac{1}{1 - \rho_1^2} \, \frac{\pi}{3}, 
\qquad 
Q_2(\rho_1,R) := \log \Big( \frac{1}{\rho_1} \Big) \psi \big( e^{\mu(\rho_1) - \mu(R)} \big),
\end{equation*}
\begin{equation*}
Q_3(R) := \log(R-1) \psi \big( e^{\mu(R-1) - \mu(R)} \big), \qquad 
Q_4(R) := \frac{2}{\sqrt{(2R-1)(R^2-2R)}}.
\end{equation*}
For any $\rho_1$, the value $Q_1(\rho_1)$ does not depend on $R$ and $Q_2(\rho_1,R), Q_4(R)$ are decreasing functions of $R$ on $[4,+\infty)$. 
Computing the derivative of $Q_3(R)$, we see that $Q_3$ is decreasing if 
\begin{equation*}
R - e^{1-2R} \frac{R^3}{(R-1)^2} - (R^2 - R + 1) \log(R-1) < 0.
\end{equation*}
Since 
$$
R - (R^2 - R) \log(R-1) = R [1 - (R-1) \log(R-1)] < 0
$$
for $R \geq 4$, we deduce that also $Q_3$ is decreasing on $[4,+\infty)$. 
Hence, the function $R \mapsto Q(\rho_1,R)$ defined in formula~\eqref{0402.31} is decreasing on $[4,+\infty)$, 
therefore,
\[
J(R) \leq Q(\rho_1,R) \leq Q(\rho_1,4) \quad \forall R \geq 4.
\]
We fix $\rho_1 = 1/8$ and we note that $d(R) \leq d(4) < 1/500 < \rho_1 = 1/8$ for all $R \geq 4$ 
(Lemma~\ref{lemma:approx d(R)} and exact rounding with {\em interval arithmetic} give $d(4) < 0.00155$). 
We compute $Q(1/8, 4) < 1.4$, thus $J(R) < 1.4$, for $R \geq 4$.
\end{proof}

\subsection{Proof of Proposition~\ref{prop:int} for ``intermediate'' $R$}

\begin{lemma}
There holds $J(R) < 1.52$ for all $R \in\bigl[\frac{23}{16}, 4\bigr]$. 
\end{lemma}

\begin{proof}
We split the interval $\bigl[\frac{23}{16},4\bigr]$ into $M$ subintervals of length $\d_R$: 
we fix 
\begin{equation} \label{0902.1}
\d_R := 2^{-10}, \quad 
\mR_n := \Big[ \frac{23}{16} + (n-1) \d_R, \frac{23}{16} + n \d_R \Big], \quad 
n = 1,2,\ldots,M, 
\end{equation}
where $M :=  (4 - \frac{23}{16}) \d_R^{-1} = 2624$.
Note that all the numbers $\frac{23}{16} + n \d_R$ with $n \in \N \cap [1,M]$  
are represented exactly by any standard computer.
Now for each $n=1,\ldots,M$ we estimate the set $\{ J(R) : R \in \mR_n\}$. 

Let $\mR$ be any of the intervals $\mR_n$ and call $R_{min}, R_{max}$ its extremal points.  
Let $\a,\b$, with $\a < \b$, be two rational numbers that are represented exactly by the computer, 
such that $d(R) \in [\a,\b]$, for all $R \in \mR$. 
To estimate $J(R)$ for all $R \in \mR$, 
we split the integration interval $[d(R),R]$ into many subintervals.

First, let $\d_0 = 2^{-m_0}$, $m_0 \in \N$, be such that 
\begin{equation} \label{0902.2}
\beta + \d_0 \leq 1, 
\end{equation}
so that $d(R) + \d_0 \leq 1$ for all $R \in \mR$. 
Since $f(\rho,R)$ is convex in $\rho$ (Lemma~\ref{lemma:convex f}),
the integral over $[d(R), d(R)+\d_0]$ is bounded by the area of the trapezoid,
\begin{equation*}
\int_{d(R)}^{d(R)+\d_0} f(\rho,R) \, d\rho 
\leq \frac{\d_0}{2}\, \Big( \frac{\sqrt{3}}{d(R)}\, + f(d(R)+\d_0,R) \Big)
=: J_0(R).
\end{equation*}
Next, since $\a \leq d(R)$, we have
\begin{equation*}
\int_{d(R)+\d_0}^1 f(\rho,R) \, d\rho 
\leq \int_{\a + \d_0}^1 f(\rho,R) \, d\rho
\end{equation*}
and $[\a + \d_0,1]$ is contained in the domain $[d(R),R)$ of $f(\cdot,R)$ for all $R \in \mR$ if 
\begin{equation*}
\b \leq \a + \d_0. 
\end{equation*}
We define $m_0$ as the integer part $[-\log_2(\b-\a)]$, 
so that $\b-\a \leq \d_0 < 2(\b-\a)$. 
Inequality~\eqref{0902.2} holds if $\a,\b$ satisfy
\begin{equation*}
3 \b - 2 \a \leq 1.
\end{equation*}
We split the integration interval $[\a + \d_0,1]$ into $N_1$ subintervals 
of length $\d_1 = 2^{-10}$ and, possibly, one smaller interval.
Let $N_1$ be the integer part of $(1 - \a - \d_0) \d_1^{-1}$ and define
\begin{equation*}
x_k := (\a + \d_0) + (k-1) \d_1, \quad k = 1, \ldots, N_1 + 1.
\end{equation*}
Thus, $x_1 = \a+\d_0$ and $1 - \d_1 < x_{N_1 + 1} \leq 1$. 
Note that the numbers $x_k$ are exactly representable by the computer. 
Since $f(\cdot,R)$ is convex, one has 
\begin{equation} \label{0902.4} 
\int_{\a+\d_0}^1 f(\rho,R) \, d\rho 
\leq J_1(R) + \int_{x_{N_1 + 1}}^1 f(\rho,R) \, d\rho
\end{equation}
where
\begin{equation*}
J_1(R) := \frac{\d_1}{2}\, \Big( f(x_1, R) + f(x_{N_1 + 1}, R) \Big)
+ \d_1 \sum_{k=2}^{N_1} f(x_k, R) .
\end{equation*}
Next, let $\d_4 > 0$ be such that $R_{min} - \d_4 \geq 1$. 
Bound~\eqref{0402.30} gives
\begin{equation*}
\int_{R - \d_4}^R f(\rho,R) \, d\rho 
< \frac{2 \sqrt{\d_4}}{\sqrt{(2R -\d_4)[(R-\d_4)^2 - 1]}} \, =: J_4(R).
\end{equation*}
Since $R \leq R_{max}$,  
\begin{equation*}
\int_1^{R - \d_4} f(\rho,R) \, d\rho 
\leq \int_1^{R_{max} - \d_4} f(\rho,R) \, d\rho, 
\end{equation*}
and $[1,R_{max} - \d_4]$ is contained in the domain $[d(R),R)$ of $f(\cdot,R)$ for all $R \in \mR$ if 
\begin{equation*}
\d_4 > R_{max} - R_{min} = \d_R. 
\end{equation*}
Since $\d_R = 2^{-10}$ and $R_{min} \geq \frac{23}{16}$, we can fix $\d_4 = 2^{-6}$. 
We split the integration interval $[1, R_{max} - \d_4]$ into $N_3$ subintervals 
of length $\d_3 = 2^{-6}$ and, possibly, one smaller interval.
Let $N_3$ be the integer part of $(R_{max} - \d_4 - 1) \d_3^{-1}$ and define
\begin{equation*}
z_k := (R_{max} - \d_4) - (k-1) \d_3, \quad k = 1, \ldots, N_3 + 1.
\end{equation*}
Thus $z_1 = R_{max} - \d_4$ and $1 \leq z_{N_3 + 1} < 1 + \d_3$.
Note that the numbers $z_k$ are exactly representable on the computer. 
Since $f(\cdot,R)$ is convex, 
\begin{equation} \label{0902.5}
\int_1^{R_{max} - \d_4} f(\rho,R) \, d\rho 
\leq J_3(R) + \int_1^{z_{N_3 + 1}} f(\rho,R) \, d\rho,
\end{equation}
where
\begin{equation*}
J_3(R) := \frac{\d_3}{2}\, \Big( f(z_1, R) + f(z_{N_3 + 1}, R) \Big)
+ \d_3 \sum_{k=2}^{N_3} f(z_k, R).
\end{equation*}
Regarding the two remaining integrals~\eqref{0902.4} and~\eqref{0902.5}, one has 
\begin{equation*}
\int_{x_{N_1 + 1}}^1 f(\rho,R) \, d\rho
+ \int_1^{z_{N_3 + 1}} f(\rho,R) \, d\rho
= \int_{x_{N_1 + 1}}^{z_{N_3 + 1}} f(\rho,R) \, d\rho
\leq J_2(R)
\end{equation*}
where
\begin{equation*}
J_2(R) := \frac{z_{N_3 + 1} - x_{N_1 + 1}}{2} \Big( f(x_{N_1 + 1},R) + f(z_{N_3 + 1},R) \Big).
\end{equation*}
Thus, 
\begin{equation} \label{0802.1}
J(R) \leq \sum_{i=0}^4 J_i(R).
\end{equation}
We use {\em interval arithmetic} to find a correct rounding for each term $J_i(R)$, uniformly in $R \in \mR$ and we repeat the procedure for all intervals in formula~\eqref{0902.1}. 
We get $J(R) \leq 1.51734 < 1.52$ for all $R \in\bigl[\frac{23}{16},4\bigr]$. 
The maximum value of the upper bound for the right--hand side of inequality~\eqref{0802.1} is obtained on the interval $\mR_n$ with $n = 490$, corresponding to $R$ close to 1.915 
(but there is no reason for which the true function $J(R)$ and the upper bound we have computed should have a maximum at the same point).
Along the procedure the exponent $m_0$ of $\d_0$ assumes integer values between 5 and 17.\\ 
The computation takes about 18 minutes on our standard computers. The code we have used is in Appendix~\ref{sec:codes}.
\end{proof}

\section{Some consequences and open questions}
\label{final}

The arguments of Section~\ref{sec:pof} can be used to prove the following proposition, also conjectured in~\cite[Conjecture~3.22]{haettenschweiler}.

\begin{prop}
Any region of a regular shrinking network bounded by only two curves must contain the origin in its interior.
\end{prop}
\begin{proof}
Such a region cannot be strictly convex, otherwise, by the shrinkers
equation the conclusion is immediate. If one of the curves is a
segment, by reflecting the region with respect to the straight line
containing such segment, which must pass through the origin, one would
obtain a $\Theta$--shrinker which is excluded by
Theorem~\ref{thm:main}. Hence, we suppose that the closed region is bounded by two curves
with the ``same convexity'' with respect to the origin, 
counterclockwise parametrized by arclength (hence going from the
triple junction $B$ to the triple junction $A$) and we assume that
the origin is outside the region.\\
As before, we separate
the analysis in two cases, according to the position of the origin
with respect to the straight line containing the segment
$\overline{AB}$. The case when the origin is below or belongs to such
line follows as in {\em Case~1} of the proof of Theorem~\ref{thm:main}
in Section~\ref{sec:pof}. The case when the origin is above the the
straight line containing the segment $\overline{AB}$, can be treated
with the same (energetic) arguments of {\em Case~2}.
\end{proof} 

As an immediate consequence, not only a $\Theta$--shrinker cannot
exist, but also there are no shrinking networks with more than one
region bounded by only two curves.\\
This is another step in the classification of the whole family of
shrinkers. For instance, it is easy to show that every bounded region
must have less than six bounding curves and its area is determined by
such number.\\
Several other results in this direction were obtained
in~\cite{schn-schu,schnurerlens,haettenschweiler}. We want to mention
the work of H\"attenschweiler~\cite{haettenschweiler} where it is
proposed the very interesting Conjecture~3.26 that there is an upper
bound for the possible number of bounded regions of a shrinker. This
clearly would imply that the possible topological structures of
compact shrinkers are finite.

\appendix
\section{Codes for interval arithmetic computations}
\label{sec:codes}
This is the code we have used to compute a correct rounding (in particular, an exact upper bound) for the right--hand side of~\eqref{0802.1}. 
  
\begin{tiny}  
\begin{verbatim}
tic ();
output_precision(6) 
delta_R = 2^(-10);
delta_1 = 2^(-10);
delta_3 = 2^(-6); 
delta_4 = 2^(-6); 
NN = 10; %% Number of iterations for the computation of d(R).
if (delta_4 < delta_R) 
  printf("Attention: it is delta_4 < delta_R, whereas it should be >=. \n")
endif

M = (4 - 23/16) / delta_R; 
J_0 = infsup(zeros(M,1)); 
J_1 = infsup(zeros(M,1));
J_2 = infsup(zeros(M,1));
J_3 = infsup(zeros(M,1));
J_4 = infsup(zeros(M,1));
J = infsup(zeros(M,1));
vec_R = infsup(zeros(M,1));
vec_dR = infsup(zeros(M,1));
m_0 = infsup(zeros(M,1));

for n=1:M
R_left = 23/16 + (n-1) * delta_R;
R_right = R_left + delta_R;
R = infsup( R_left , R_right ); % R is an interval temporary variable
vec_R(n) = R; % the value of R is stored in the vector vec_R at row n 

%% Here we calculate d(R) as a correctly rounding interval:
b = R^2 /2 - log(R) + (1/2) * log(infsup(3)) - log(infsup(2)); 
a = -b ;
A = infsup(zeros(NN,1)) ;
A(1) = a ;
for k = 2:NN
  A(k) = a + (1/2) * exp(2 * A(k-1)) ;
  endfor
alpha = inf(exp(A(NN))) ;
B = infsup(zeros(NN,1)) ;
B(1) = a + 1/2 ;
for k = 2:NN
  B(k) = a + (1/2) * exp(2 * B(k-1)) ;
  endfor
beta = sup(exp(B(NN))) ;
if (beta <= alpha)
  printf("Attention: alpha >= beta at row:\n")
  n
endif
if (3*beta - 2*alpha > 1)
  printf("Attention: 3 alpha - 2 beta > 1 at row:\n")
  n
endif
% Verify: it should be: 
% inf(ver_a) > 0  and  sup(ver_b) < 0.
ver_a = 1/2 * infsup(alpha)^2 - log(infsup(alpha)) - 1/2 * R^2 + log(R) - log(sqrt(infsup(3)) / 2);
ver_b = 1/2 * infsup(beta)^2 - log(infsup(beta)) - 1/2 * R^2 + log(R) - log(sqrt(infsup(3)) / 2);
if (inf(ver_a) <= 0) 
  printf("Attention: not accurate enclosure of d(R) at row:\n")
  n
endif
if (sup(ver_b) >= 0) 
  printf("Attention: not accurate enclosure of d(R) at row:\n")
  n
endif
dR = infsup( alpha , beta );
vec_dR(n) = dR;
m_0(n) = floor( - log2(beta - alpha) );
delta_0 = 2^(-m_0(n));
if (delta_0 < beta - alpha)
  printf("Attention: something wrong, delta_0 too small at row:\n")
  n
endif  
if (alpha == 0)
  printf("Attention: alpha = 0 at row:\n")
  n
endif  
if (beta + delta_0 > 1)
  printf("Attention: beta + delta_0 > 1 at row:\n")
  n
  printf("You should take a larger NN and/or a smaller delta_R")
  endif
%%%%%%%%%%%%%%%%%%%%%%%%%%%%%%%%%%%%%%
% J_0 
p = dR + delta_0;
fp = exp( p^2 / 2 - R^2 / 2) * R * (1/p)^2 * rsqrt( 1 - exp( p^2 - R^2 ) * R^2 * (1/p)^2 );
J_0(n) = delta_0 / 2 * ( sqrt(infsup(3)) / dR + fp );
%%%%%%%%%%%%%%%%%%%%%%%%%%%%%%%%%%%%%%
% J_1
N_1 = floor( (1 - alpha - delta_0) / delta_1 );
x = zeros(N_1 + 1,1);
x(1) = alpha + delta_0;
for i = 2 : N_1+1
  x(i) = x(i-1) + delta_1;
  endfor
if (x(N_1 + 1) > 1) 
  printf("Attention: x(N_1 + 1) > 1 at row:\n") 
  n
  endif 
y = infsup(zeros(N_1+1,1));
y = exp( x.^2 / 2 - R^2 / 2) * R .* ((1./x).^2) .* rsqrt( 1 - exp( x.^2 - R^2 ) * R^2 .* ((1./x).^2) );
J_1(n) = delta_1 * (y(1) + y(N_1+1)) / 2 + delta_1 * sum( y(2 : N_1) );
%%%%%%%%%%%%%%%%%%%%%%%%%%%%%%%%%%%%%%
%%% J_4
J_4(n) = 2 * sqrt(infsup(delta_4)) * rsqrt(2*R - delta_4) * rsqrt((R - delta_4)^2 - 1);
%%%%%%%%%%%%%%%%%%%%%%%%%%%%%%%%%%%%%%
%%% J_3
N_3 = floor( (sup(R) - delta_4 - 1) / delta_3 );
z = zeros(N_3+1,1);
z(1) = sup(R) - delta_4;
for i = 2 : N_3+1
  z(i) = z(i-1) - delta_3;
endfor
if ( z(N_3 + 1) < 1 ) 
  printf("Attention: z(N_3 + 1) < 1 at row:\n") 
  n
  endif 
w = infsup(zeros(N_3+1,1));
w = exp( z.^2 / 2 - R^2 / 2) * R .* ((1./z).^2) .* rsqrt( 1 - exp( z.^2 - R^2 ) * R^2 .* ((1./z).^2) );
J_3(n) = delta_3 * (w(1) + w(N_3 + 1)) / 2 + delta_3 * sum( w(2 : N_3) );
%%%%%%%%%%%%%%%%%%%%%%%%%%%%%%%%%%%%%%
%%% J_2 
if ( z(N_3 + 1) - x(N_1 + 1) >= delta_1 + delta_3 )
  printf("Attention: z(N_3 + 1) - x(N_1 + 1) should not be so large, at row:\n")
  n
endif
J_2(n) = ( y(N_1 + 1) + w(N_3 + 1) ) * ( z(N_3 + 1) - x(N_1 + 1) ) / 2;
%%%%%%%%%%%%%%%%%%%%%%%%%%%%%%%%%%%%%%
endfor
J = J_0 + J_1 + J_2 + J_3 + J_4;
[J_max, n_max] = max(sup(J));
elapsed_time = toc ();

M
printf("J_0,J_1,J_2,J_3,J_4 at max:\n")
[J_0(n_max); J_1(n_max); J_2(n_max); J_3(n_max); J_4(n_max)]

printf("R, d(R) at max:\n")
[vec_R(n_max); vec_dR(n_max)]

printf("Time in seconds:\n")
tt = elapsed_time;
tt % tt/60

printf("Min and max of d(R):\n")
min(inf(vec_dR))
max(sup(vec_dR))

printf("Min and max of m_0:\n")
min(m_0)
max(m_0)

printf("J max:\n")
J(n_max)
\end{verbatim}
\end{tiny}

\bibliographystyle{amsplain}
\bibliography{networkbib}

\providecommand{\bysame}{\leavevmode\hbox to3em{\hrulefill}\thinspace}
\providecommand{\MR}{\relax\ifhmode\unskip\space\fi MR }
\providecommand{\MRhref}[2]{%
  \href{http://www.ams.org/mathscinet-getitem?mr=#1}{#2}
}
\providecommand{\href}[2]{#2}
\begin{thebibliography}{10}

\bibitem{ablang1}
U.~Abresch and J.~Langer, \emph{The normalized curve shortening flow and
  homothetic solutions}, J. Diff. Geom. \textbf{23} (1986), no.~2, 175--196.

\bibitem{BeNo}
G.~Bellettini and M.~Novaga, \emph{Curvature evolution of nonconvex
  lens--shaped domains}, J. Reine Angew. Math. \textbf{656} (2011), 17--46.

\bibitem{brakke}
K.~A. Brakke, \emph{The motion of a surface by its mean curvature}, Princeton
  University Press, NJ, 1978.

\bibitem{chenguo}
X.~Chen and J.-S. Guo, \emph{Self--similar solutions of a 2--{D}
  multiple--phase curvature flow}, Phys. D \textbf{229} (2007), no.~1, 22--34.

\bibitem{epswei}
C.~L. Epstein and M.~I. Weinstein, \emph{A stable manifold theorem for the
  curve shortening equation}, Comm. Pure Appl. Math. \textbf{40} (1987), no.~1,
  119--139.

\bibitem{FefSec96}
C.~L. Fefferman and L.~A. Seco, \emph{Interval arithmetic in quantum
  mechanics}, Applications of interval computations ({E}l {P}aso, {TX}, 1995),
  Appl. Optim., vol.~3, Kluwer Acad. Publ., Dordrecht, 1996, pp.~145--167.

\bibitem{Fousse2007}
L.~Fousse, G.~Hanrot, V.~Lef{\`e}vre, P.~P{\'e}lissier, and P.~Zimmermann,
  \emph{M{PFR}: a multiple--precision binary floating--point library with
  correct rounding}, ACM Trans. Math. Software \textbf{33} (2007), no.~2, Art.
  13, 15.

\bibitem{haettenschweiler}
J.~H{\"a}ttenschweiler, \emph{Mean curvature flow of networks with triple
  junctions in the plane}, Master's thesis, ETH Z\"urich, 2007.

\bibitem{conforming}
O.~Heimlich, \emph{{GNU} {O}ctave {I}nterval {P}ackage {M}anual -- conformance
  claim to the ieee standard ieee},
  http:/$\!\!$/$\!$octave.sourceforge.net/interval/package\_doc/IEEE-Std-1788\_002d2015.html.

\bibitem{huisk3}
G.~Huisken, \emph{Asymptotic behavior for singularities of the mean curvature
  flow}, J. Diff. Geom. \textbf{31} (1990), 285--299.

\bibitem{IEEE1788}
IEEE, \emph{{IEEE} {S}tandard for {I}nterval {A}rithmetic, {IEEE 1788--2015}},
  http:/$\!\!$/$\!$standards.ieee.org/findstds/standard/1788-2015.html.

\bibitem{ilman1}
T.~Ilmanen, \emph{Elliptic regularization and partial regularity for motion by
  mean curvature}, Mem. Amer. Math. Soc., vol. 108(520), Amer. Math. Soc.,
  1994.

\bibitem{ilman3}
\bysame, \emph{Singularities of mean curvature flow of surfaces},
  http:/$\!\!$/$\!$www.math.ethz.ch/{$\sim$}ilmanen/papers/sing.ps, 1995.

\bibitem{Ilnevsch}
T.~Ilmanen, A.~Neves, and F.~Schulze, \emph{On short time existence for the
  planar network flow}, ArXiv Preprint Server -- http:/$\!\!$/$\!$arxiv.org,
  2014.

\bibitem{MMN13}
A.~Magni, C.~Mantegazza, and M.~Novaga, \emph{Motion by curvature of planar
  networks {II}}, ArXiv Preprint Server -- http:/$\!\!$/$\!$arxiv.org, to
  appear on Ann. Sc. Norm. Sup. Pisa, 2013.

\bibitem{mannovplu}
C.~Mantegazza, M.~Novaga, and A.~Pluda, \emph{Motion by curvature of networks
  with two triple junctions}, in preparation.

\bibitem{mannovplusch}
C.~Mantegazza, M.~Novaga, A.~Pluda, and F.~Schulze, \emph{Evolution of networks
  with multiple junctions}, in preparation.

\bibitem{mannovtor}
C.~Mantegazza, M.~Novaga, and V.~M. Tortorelli, \emph{Motion by curvature of
  planar networks}, Ann. Sc. Norm. Sup. Pisa \textbf{3 (5)} (2004), 235--324.

\bibitem{mazsae}
R.~Mazzeo and M.~S{\'a}ez, \emph{Self--similar expanding solutions for the
  planar network flow}, Analytic aspects of problems in {R}iemannian geometry:
  elliptic {PDE}s, solitons and computer imaging, S\'emin. Congr., vol.~22,
  Soc. Math. France, Paris, 2011, pp.~159--173.

\bibitem{Moore91}
R.~E. Moore, \emph{Interval tools for computer aided proofs in analysis},
  Computer aided proofs in analysis ({C}incinnati, {OH}, 1989), IMA Vol. Math.
  Appl., vol.~28, Springer, New York, 1991, pp.~211--216.

\bibitem{pluda}
A.~Pluda, \emph{Evolution of spoon--shaped networks}, ArXiv Preprint Server --
  http:/$\!\!$/$\!$arxiv.org, to appear on Netw. Heterog. Media, 2015.

\bibitem{schnurerlens}
O.~C. Schn{\"u}rer, A.~Azouani, M.~Georgi, J.~Hell, J.~Nihar, A.~Koeller,
  T.~Marxen, S.~Ritthaler, M.~S{\'a}ez, F.~Schulze, and B.~Smith,
  \emph{Evolution of convex lens--shaped networks under the curve shortening
  flow}, Trans. Amer. Math. Soc. \textbf{363} (2011), no.~5, 2265--2294.

\bibitem{schn-schu}
O.~C. Schn\"urer and F.~Schulze, \emph{Self--similarly expanding networks to
  curve shortening flow}, Ann. Sc. Norm. Super. Pisa \textbf{6} (2007), no.~4,
  511--528.

\end{thebibliography}

\end{document}